\documentclass[11pt]{amsart}
\usepackage{amsmath}
\usepackage{CJK}
\allowdisplaybreaks
\usepackage{amssymb,latexsym}
\usepackage{url}
\usepackage{epsfig}
\usepackage[matrix,arrow]{xy}
\usepackage[usenames]{color}
\usepackage{slashed}
\usepackage{parskip}
\usepackage{amsmath,amssymb,graphics}
\usepackage{mathrsfs}
\usepackage{graphicx}

\newcommand{\mathsym}[1]{{}}

\newtheorem{thm}{Theorem}[section]

\newtheorem{lem}[thm]{Lemma}
\newtheorem{prop}[thm]{Proposition}

\theoremstyle{definition}
\newtheorem{defn}{Definition}[section]
\numberwithin{equation}{section}
\newtheorem{rem}{Remark}[section]
\theoremstyle{example}

\newcommand{\be}{\begin{equation}}
\newcommand{\ee}{\end{equation}}
\newcommand{\bag}{\begin{eqnarray}}
\newcommand{\eag}{\end{eqnarray}}
\newcommand{\ban}{\begin{eqnarray*}}
\newcommand{\ean}{\end{eqnarray*}}
\newcommand{\ba}{\begin{aligned}}
\newcommand{\ea}{\end{aligned}}

\newcommand{\bpf}{\begin{proof} }
\newcommand{\epf}{\end{proof} }

\begin{document}
\title{Positivity in Foliated Manifolds and Geometric Applications}
\author{Yashan Zhang}
\address{Institute of Mathematics, Hunan University, Changsha 410082, China}
\email{yashanzh@hnu.edu.cn}
\author{Tao Zheng}
\address{School of Mathematics and Statistics, Beijing Institute of Technology, Beijing 100081, China}
\email{zhengtao08@amss.ac.cn}
\thanks{Zhang is partially supported by Fundamental Research Funds for the Central Universities (No. 531118010468) and National Natural Science Foundation of China (No. 12001179)}
\thanks{Zheng is partially supported by Beijing Institute of Technology Research Fund Program for Young Scholars}
\begin{abstract}
We introduce the notion of positivity for a real basic $(1,1)$ class in basic Bott-Chern cohomology group on foliated manifolds, and study the relationship between this positivity and the negativity of transverse holomorphic sectional curvature and give some geometric applications.
\end{abstract}
\maketitle

\section{Introduction}
Since there is no general method to solve differential equations
even in $\mathbb{R},$
mathematicians rather try to study the geometrical and topological properties of
global manifolds and their asymptotic behaviors.   The exact purpose of
foliation theory is the qualitative study of differential equations which was initiated by
the works of  Poincar\'e and  Bendixson, and developed later by  Ehresmann,
 Reeb,  Haefliger, and many others. Since then the subject has been a
wide field in mathematical research and there are still many  open questions in some directions
in the theory of foliations (see, for example, \cite{elk2014} and references therein).

The notion of transverse structure plays a key role in the study and classification of foliations. This is a research object of global analysis, and we refer to \cite{elk2014} and references therein for a detailed survey. Since El Kacimi-Alaoui \cite{elk90} proves the transverse Calabi-Yau theorem, there are many transverse counterparts on foliated manifolds (especially Sasakian manifolds) of the famous results on complex manifolds (especially the K\"ahler manifolds), such as  the existence of canonical metrics on Sasakian manifolds \cite{fow09}, Sasaki-Einstein metrics and K-(semi-)stability on Sasakian manifolds \cite{cs15,cs18},   the Frankel conjecture on Sasakian manifolds \cite{hs15,hs16}, the Uhlenbeck-Yau theorem \cite{uhlenbeckyau} about the existence of Hermitian-Einstein structure \cite{bs2010}, foliated Hitchin-Kobayashi correspondence \cite{baragliahekmati2018},  the geometric pluripotential theory \cite{hl18} on Sasakian manifolds, the transverse fully nonlinear equations \cite{fengzhengsasaki1} corresponding to \cite{twjams10,twjams,stw1503}, and the Higgs bundle on foliated manifolds \cite{wuzhangscm}.

In this paper, we introduce notions of the positivity of basic $(1,1)$ classes in the Bott-Chern cohomological group $H^{1,1}_{\mathrm{BC}}(X/\mathcal{F},\mathbb{R})$ using the transverse invariant measure. Then we consider the relationship between this positivity and negativity of transverse holomorphic sectional curvature and also give some geometric applications.

A fundamental conjecture of Yau in 1970's predicts that a compact K\"ahler manifold admitting a K\"ahler metric of negative holomorphic sectional curvature has an ample canonical line bundle, which has been comfirmed by Wu-Yau \cite{wuyauinventiones} and Tosatti-Yang \cite{tosattiyangjdg} (also see \cite{DTjdg,wuyaucag,YZtams,Zmz,ZZ} for some further developments). Our first theorem is a transverse version of the works of Wu-Yau \cite{wuyauinventiones} and Tosatti-Yang \cite{tosattiyangjdg}.
\begin{thm}
\label{tmainthmtyjdg}
Let $(X,\mathcal{F})$ be a closed oriented$,$ taut$,$ transverse K\"ahler foliated manifold$,$ where $\mathcal{F}$ is the foliation with complex codimension $n$ and $\omega$ is a transverse K\"ahler metric with nonpositive transverse holomorphic sectional curvature. Then the normal canonical bundle $K_{X/\mathcal{F}}$ is transverse nef. If furthermore $\omega$ is a transverse K\"ahler metric with negative transverse holomorphic sectional curvature, then there exists a smooth basic function $u\in C^\infty(X/\mathcal{F},\mathbb{R})$ such that $ \omega_u :=-\mathrm{Ric}(\omega)+\sqrt{-1}\partial\bar\partial u$ is the transverse K\"ahler-Einstein metric  with $\mathrm{Ric}(\omega_u)=-\omega_u$.
\end{thm}

As an application of Theorem \ref{tmainthmtyjdg} and results of Touzet \cite[Theorem 1.2]{touzettoulouse2010}, we are able to prove the following
\begin{thm}\label{mainthm2}
Suppose that $X$ is a compact K\"ahler manifold and $\mathcal{F}$ is (regular) holomorphic with complex codimension $n$. Assume also that $c_1(T_{\mathcal{F}})=0$ and that there exists a transverse K\"ahler metric with  negative transverse holomorphic sectional curvature.  Then the lift of $\mathcal{F},$ up to some finite covering of $X,$ is defined by a locally trivial fibration over a manifold whose first Chern class is quasi-negative$;$ indeed$,$ $\mathcal{F}$ is defined by the Iitaka-Kodaira fibration of $X$.
\end{thm}

The outline of the paper is as follows. In Section \ref{section:preliminaries}, we collect preliminaries of global analytic and geometric aspects  of foliated manifolds. In section \ref{section:mainthm}, we study the geometric partial differential equations on foliated manifolds and as its geometric applications, we prove Theorem \ref{tmainthmtyjdg} and Theorem \ref{mainthm2}. In Appendix \ref{section:appendix}, we collect preliminaries for distribution and current in order to deduce the transverse versions of the Poincar\'e lemma and  the Dolbeault-Grothendieck lemma.

After the first version of this paper was posted on the arXiv, we learned that the Sasakian case of Theorem \ref{tmainthmtyjdg} was independently proved by Yong Chen, the Wu-Yau theorem on Sasakian manifolds, arXiv:2109.05414.

\noindent {\bf Acknowledgements}
The authors thank professors Huai-Dong Cao, Jean-Pierre Demailly, Valentino Tosatti and Ben Weinkove for invaluable directions. The authors also thank Professor David Baraglia and Pedram Hekmati for discussing the partition of unity for basic functions and examples of foliated manifolds with  taut foliation, and Professor Sheng Rao and Dr. Runze Zhang for pointing out the mistake in defining the notion of the positivity of basic $(p,p)$ forms and currents.

\section{Preliminaries}
\label{section:preliminaries}
In this section, we collect some preliminaries of global analytic and geometric aspects of foliated manifolds (see \cite{barreelkacimi-alouifoliations,elk2014} for example).
In what follows, Greek indices, Latin indices and capital Latin indices run from $1$ to $r$, $1$ to $n$ and $1$ to $k$ respectively, unless otherwise indicated.
\subsection{Foliation}
\begin{defn}
\label{kacimidefn2.1.1}
Let $X$ be a smooth manifold with $\dim_{\mathbb{R}}X=k+n$.  Then a foliation $\mathcal{F}$ on $X$ with real dimension $k$ and real codimension $n$
is defined by an atlas $\mathscr{A}$  on $X$ consisting of homeomorphism $\kappa$ of open set $U_\kappa\subset X$ to $\tilde V_\kappa\times\tilde U_\kappa\subset\mathbb{R}^{k }\times \mathbb{R}^n$  such that
$$
\kappa\circ\kappa'^{-1}:\;\kappa'(U_{\kappa}\cap U_{\kappa'})
\to \kappa(U_{\kappa}\cap U_{\kappa'}),\quad (t,x)\mapsto(s,y)=(s(t,x),y(x)),
$$
with $(t,x),\,(s,y)\in \mathbb{R}^k\times\mathbb{R}^n$.

This kind of local coordinate patch $(U_\kappa,\kappa)$ is called to be distinguished for the foliation $\mathcal{F}$.
\end{defn}
We denote
by $T_{\mathcal{F}}$  the tangent bundle to $\mathcal{F}$, and by $\nu\mathcal{F}$ the quotient $T_X/T_{\mathcal{F}}$, which is the
normal bundle to $\mathcal{F}$. Let $\mathfrak{X}(\mathcal{F})$  denote the space of all the smooth sections of $T_{\mathcal{F}}$.
If $g$ is a Riemannian metric on $X$, then $(X,g,\mathcal{F})$ is called a Riemannian foliated manifold \cite{vaisman}. In the distinguished coordinate patch $(U;t,x)$, we set
\begin{equation}
e_i:=\frac{\partial}{\partial x^i}-\sum_{P=1}^{ k}A_i^P\frac{\partial}{\partial t^P},\quad 1\leq i\leq n,
\end{equation}
such that
\begin{equation}
g\left(\frac{\partial}{\partial t^P},e_i\right)=0,\quad 1\leq i\leq n,\quad 1\leq P\leq k.
\end{equation}
A direct calculation yields that
\begin{equation}
A:=\sum_{1\leq i\leq n,\,1\leq P\leq k}A_i^P\frac{\partial}{\partial t^P}\otimes \mathrm{d}x^i
\end{equation}
is a well-defined tensor and that
\begin{equation}
T_X:=T_{\mathcal{F}}\oplus T_{\mathcal{F}}^{\perp},
\end{equation}
with respect to the Riemannian metric $g$, where
\begin{equation}
T_{\mathcal{F}\upharpoonright U}=\mathrm{Span}_{\mathbb{R}}\left\{\frac{\partial}{\partial t^1},\cdots,\frac{\partial}{\partial t^k}\right\},\quad
 T_{\mathcal{F}\upharpoonright U}^{\perp}
:=\mathrm{Span}_{\mathbb{R}}\{e_1,\cdots,e_n\}.
\end{equation}
Note that $\nu\mathcal{F}$ is smoothly isomorphic to $T_{\mathcal{F}}^{\perp}$. Let \begin{equation}
\label{defnthetaalpha}
\theta^P:=\mathrm{d}t^P+A^P_i \mathrm{d}x^i,\quad 1\leq P\leq k.
\end{equation}
Then $\{\mathrm{d}x^i,\theta^P\}$ is the dual of $\{e_i,\frac{\partial}{\partial t^P}\}$.

A differential form $u$ is called of the type $\{p,q\}$ and order $s\in \mathbb{N}\cup\{\infty\}$ if in the distinguished local coordinate patch $(U;t,x)$ it has the expression
\begin{equation}
u=\frac{1}{p!q!}u_{i_1,\cdots,i_p;P_1,\cdots,P_q}\mathrm{d}x^{i_1}\wedge
\cdots\wedge\mathrm{d}x^{i_p}\wedge\theta^{P_{1}}\wedge\cdots\wedge\theta^{P_q},
\end{equation}
where $u_{i_1,\cdots,i_p;P_1,\cdots,P_q}\in C^s(U,\mathbb{C})$  is skew-symmetric separately in the indices $i$ and the indices $P$.

Each differential $r$-form has unique decomposition as a sum of form of the type $\{p,q\}$ with $p+q=r$ and the same order (see for example \cite{vaisman}).
Let ${}^s\mathscr{E}_{\mathcal{F}}^p(X)$ (resp. ${}^s\mathscr{D}_{\mathcal{F}}^p(X)$) denote the set of the differential forms of type $\{p,k\}$ and of order $s$ (resp. with compact support).

A differential $r$ form $\varphi\in\Omega^r(X)$ (resp. a $r$ current $T\in{}^s\mathscr{D}'^r(X)$) is said to be basic if it satisfies $i_\xi\varphi=\mathcal{L}_\xi\varphi=0$ (resp. $i_\xi T=\mathcal{L}_\xi T=0$) for each $\xi\in\mathfrak{X}(\mathcal{F})$.  Let $\Omega^r(X/\mathcal{F})$ (resp. ${}^s\mathscr{D}'^r(X/\mathcal{F})$) denote
the space of basic forms (resp. basic current) of degree $r$ on the foliated manifold $(X,\mathcal{F})$. In particular, a function $f:\;X\to \mathbb{C}$ is called basic if $f(\varphi_\xi(t))$ is independent of $t$, where $\varphi_\xi(t)$ is the integral curve of $\xi$ for each $\xi\in\mathfrak{X}(\mathcal{F})$.  Let $C^k (X/\mathcal{F},  \mathbb{C})$ denote the set of basic functions in $C^k(X, \mathbb{C})$ with $k\in \mathbb{N} \cup\{\infty\}$.
\begin{lem}\label{lembasicfunctionwithcompactsupport0}
Let $X$ be a smooth manifold with $\dim_{\mathbb{R}}X=k+n$, and $\mathcal{F}$ a foliation on $X$ with real dimension $k$. Then any basic function defined on a distinguished chart $(U,\kappa)$ with compact support is zero.
\end{lem}
\begin{proof}
Without loss of generality, we write
$$
\kappa:\;U\to \tilde{V}\times \tilde{U}\subset \mathbb{R}^k\times \mathbb{R}^n.
$$
Then the conclusion follows from the fact that
$$
\partial (\tilde V\times \tilde U)=\left((\partial \tilde V)\times \tilde U\right)\bigcup \left(\tilde V\times (\partial \tilde U)\right).
$$
\end{proof}
It follows from Lemma \ref{lembasicfunctionwithcompactsupport0} that there does not exist a partition of unity subordinate to a given cover by basic functions (See for example \cite{baragliahekmati2018}).

Note that a basic current $T\in {}^s\mathscr{D}'^{n-r}(X/\mathcal{F})$ is a continuous linear functional
$$
T:\;{}^s\mathscr{D}_{\mathscr{F}}^r(X)\to\mathbb{C}
$$
with $\mathcal{L}_\xi T=0$ for each $\xi\in\mathfrak{X}(\mathcal{F})$ (see for example \cite{craioveanuputa1980,vaisman}).

In the distinguished coordinate patch $(U;t^1,\cdots,t^k,x^1,\cdots,x^n)$, a basic $p$ current $T$ can be written as
 $$
 T= \sum_{i_1<\cdots<i_p} T_{i_1,\cdots,i_p}\mathrm{d}x^{i_1}\wedge\cdots\wedge\mathrm{d}x^{i_p},
 $$
 where $\frac{\partial T_{i_1,\cdots,i_p}}{\partial t^\alpha} =0$ for $1\leq \alpha\leq k$. Hence it follows from Theorem \ref{tthmbasiccurrentcohomology} that the basic cohomology group $H^*(X/\mathcal{F})$ is given by
\begin{equation}
H^p(X/\mathcal{F},\mathbb{R}):=\frac{\{\varphi\in\Omega^p(X/\mathcal{F}):\mathrm{d}\varphi=0\}}{\mathrm{d}\Omega^{p-1}(X/\mathcal{F})}
\simeq\frac{\{T\in\mathscr{D}'^p(X/\mathcal{F}):\mathrm{d}T=0\}}{\mathrm{d}\mathscr{D}'^{p-1}(X/\mathcal{F})}.
\end{equation}
\begin{defn}
A codimension $n$ foliation on $X$ is defined by an open cover $\mathscr{U}:=\{U_i\}_{i\in I}$ with
submersions $\tau_i:\;U_i\to T$ over a transverse manifold $T$ with $\dim_{\mathbb{R}}T=n,$  and for each
nonempty intersection $U_i\cap U_j,$  a diffeomorphism
\begin{equation*}
\gamma_{ij}:\;\tau_i(U_i\cap U_j)\to \tau_j(U_i\cap U_j)
\end{equation*}
satisfying $\tau_j(x)=\gamma_{ij}\circ \tau_i(x)$ for all $x\in U_i\cap U_j$.  We say that $\{U_i,\tau_i,T,\gamma_{ij}\}$  is a
foliated cocycle defining $\mathcal{F}$.
\end{defn}
The foliation $\mathcal{F}$ is called to be transversely orientable if $T$ can be given an
orientation preserved by all the local diffeomorphisms $\{\gamma_{ij}\}$.

The foliation $\mathcal{F}$  is called to be Riemannian if there exists a Riemannian metric on $T$
such that all the local diffeomorphisms  $\{\gamma_{ij}\}$ are isometries. Using the submersions
$f_i:\;U_i\to T$
 one can construct on $M$ a Riemannian metric which can be written
in local coordinates as
\begin{equation}
g=\sum_{\gamma=1}^k\theta^\gamma\otimes \theta^\gamma+\sum_{p,q=1}^ng_{pq}(x)\mathrm{d}x^p\otimes \mathrm{d}x^q.
\end{equation}
This metric is called bundle-like.

The foliation $\mathcal{F}$ is called to be transversely holomorphic (resp. K\"ahler) if $T$ is a complex manifold (resp. a K\"ahler manifold)
and the all the  diffeomorphisms $\{\gamma_{ij}\}$ are local biholomorphisms (resp. in addition preserving the K\"ahler form on $T$).

It follows from \cite{ekhe86} that $\dim_{\mathbb{R}}H^n(X/\mathcal{F},\mathbb{R})=0$ or $1$. If $\dim_{\mathbb{R}}H^n(X/\mathcal{F},\mathbb{R})=1$, then
  $\mathcal{F}$ is called homologically orientable, which is equivalent to the existence of a (real) volume form
on the leaves $\chi$ which is $\mathcal{F}$-relatively closed, i.e., $\mathrm{d}\chi(\xi_1,\cdots,\xi_k,\cdot)=0$ for $\xi_1,\cdots,\xi_k\in\mathfrak{X}(\mathcal{F})$  (cf. \cite{masa1992} based on \cite{sullivancmh1979,rummlercmh1979,heafligerjdg1980}).  In this case, we can complete the transverse metric
by a Riemannian metric along the leaves to obtain a Riemannian metric on the
whole manifold for which the leaves are minimal and $\chi$ is associated to this metric and we also say that $ \mathcal{F} $ is taut.
This hypothesis will enable one to define an inner product on $\Omega^r(X/\mathcal{F})$  without
using the basic manifold $T$.
\begin{lem}
[Basic Stokes' theorem \cite{baragliahekmati2018}]
\label{basicstokes}
Let $(X,\mathcal{F})$ be a closed foliated manifold with a taut foliation $\mathcal{F}$ of
$\dim_{\mathbb{R}}\mathcal{F}=k $ and codimension $n$.  Then we have
\begin{equation}
\label{bst}
\int_X(\mathrm{d}_{\mathrm{B}}\varphi)\wedge\chi=0,\quad \forall\;\varphi\in\Omega^{n-1}(X/\mathcal{F}),
\end{equation}
where $\chi $ is the volume form of leaves which is $\mathcal{F}$-relatively closed.
\end{lem}
Let $\pi:\;E\to X$ be a complex vector bundle defined by a cocycle $\{U_\kappa,g_{\kappa\kappa'},G\}$,
where $\{U_\kappa\}$ is an open cover of $X$ and $g_{\kappa\kappa'}:\;U_\kappa\cap U_{\kappa'}\to G\subset GL(n,\mathbb{C})$ are the
transition functions.
Let $\Gamma(E)$ (resp. $\Gamma_c(E)$) denote the set of all smooth sections (resp. with compact support) of $E$.

A connection on the vector bundle $E$ is a map
\begin{equation*}
\nabla:\;\mathfrak{X}(X)\times \Gamma(E)\to \Gamma(E),\quad (V,\xi)\mapsto \nabla_V\xi
\end{equation*}
such that for all $f,h,u,v\in C^\infty(X,\mathbb{C})$ and $\xi,\eta\in\Gamma(E)$, there holds
\begin{equation}
\nabla_{fV+hY}(u\xi+v\eta)=V(u)f \xi+Y(v)h \eta+fu\nabla_V\xi+fv\nabla_V\eta
+hu\nabla_Y\xi+hv\nabla_Y\eta.
\end{equation}
The curvature $\mathcal{R}$ of $\nabla$ is defined by
$$
\mathcal{R}(V,Y)Z:=\nabla_V\nabla_YZ-\nabla_Y\nabla_VZ-\nabla_{[V,Y]}Z,\quad \forall\;V,Y,Z\in\mathfrak{X}(X).
$$
It follows from \cite{kambertondeur1975} that the vector bundle $E$ is foliated if and only if it admits a linear connection such that its curvature $\mathcal{R}$ satisfies
\begin{equation}
\mathcal{R}(\xi_1,\xi_2)=0,\quad \forall\;\xi_1,\xi_2\in\mathfrak{X}(\mathcal{F}),
\end{equation}
and that
the vector bundle $E$ is an $\mathcal{F}$-bundle if and only if it admits a linear connection (called basic connection) such that its curvature $\mathcal{R}$ satisfies
\begin{equation}
\mathcal{R}(\xi,\cdot)=0,\quad \forall\;\xi\in\mathfrak{X}(\mathcal{F}).
\end{equation}
In the  \v{C}ech language, a vector bundle $E$ is foliated if and only if it can be defined by
a cocycle $\{U_\kappa,g_{\kappa\kappa'},G\}$ such that the components of $g_{\kappa\kappa'}$ are basic functions, if and only if it admits a connection $\nabla$ such that the connection form $\omega$ satisfies $i_\xi\omega=0$ for each $\xi\in\mathfrak{X}(\mathcal{F})$.  A vector bundle $E$ is the $\mathcal{F}$-bundle if and only if it admits a connection $\nabla$ such that the connection form $\omega$ is a basic $1$-form.

A section $\eta\in\Gamma(E)$ is called basic if $\nabla_\xi\eta=0$ for each $\xi\in\mathfrak{X}(\mathcal{F})$, where $\nabla$ is the basic connection on the $\mathcal{F}$-bundle. Let $\Gamma(E/\mathcal{F})$ denote the set of all basic sections of $E$.

If a vector bundle $E$ is the $\mathcal{F}$-bundle, then the dual bundle $E^*$ and all of its exterior and symmetric powers $\wedge^*E$ and $\mathcal{S}^*E$ are
the $\mathcal{F}$-bundle. In particular, $E^*\otimes \bar E^*$ is also the $\mathcal{F}$-bundle. If there exists $h\in \Gamma(E^*\otimes \bar E^*)$ such that $h$ is positive and basic with respect to the basic connection, then $E$ is called a Hermitian $\mathcal{F}$-bundle.

Let $\mathcal{F}$ be a transverse holomorphic foliation on $X$ with real dimension $k$ and complex codimension $n$, and let $\nu$ be the complexified normal bundle $\nu\mathcal{F}\otimes_{\mathbb{R}}\mathbb{C}$ of $\nu\mathcal{F}$. Let $J$ be the automorphism of $\nu$ associated to the complex structure; $J$ satisfies $J^2=-\mathrm{id}$ and then
it has two eigenvalues, $\sqrt{-1}$ and $-\sqrt{-1}$, with associated eigensubbundles denoted $\nu^{1,0}$
and $\nu^{0,1}$, respectively. The holomorphic $\mathcal{F}$-line bundle $K_{X/\mathcal{F}}:=\bigwedge^n\left(\nu^{1,0}\right)^*$ is called the normal canonical line bundle. We have a splitting $\nu^*=\left(\nu^{1,0}\right)^*\oplus\left(\nu^{0,1}\right)^*$ which gives rise to a
decomposition
\begin{equation*}
\bigwedge^{r}\nu^*=\bigoplus_{p+q=r}\bigwedge^{p,q},
\end{equation*}
where $\bigwedge^{p,q}=\bigwedge^p(\nu^{1,0})^*\otimes\bigwedge^p(\nu^{0,1})^*$. Basic sections of $\bigwedge^{p,q}$ consist of basic forms of
type $(p,q)$. They form a vector space denoted by $\Omega^{p,q}(X/\mathcal{F})$. We have
\begin{equation}
\Omega^{r}(X/\mathcal{F})=\bigoplus_{p+q=r}\Omega^{p,q}(X/\mathcal{F}).
\end{equation}
The exterior differential $\mathrm{d}$ decomposes into a sum of two operators
\begin{equation}
\partial :\;\Omega^{p,q}(X/\mathcal{F})\to\Omega^{p+1,q}(X/\mathcal{F}),\quad
\bar\partial :\;\Omega^{p,q}(X/\mathcal{F})\to\Omega^{p,q+1}(X/\mathcal{F}).
\end{equation}
A basic function $f\in C^1 (X/\mathcal{F},\mathbb{C})$ is called basic holomorphic if $\bar\partial f=0$.

For the basic currents, we also have
\begin{equation}
\mathscr{D}^{r}(X/\mathcal{F})=\bigoplus_{p+q=r}\mathscr{D}^{p,q}(X/\mathcal{F}),\quad
\mathscr{D}'^{r}(X/\mathcal{F})=\bigoplus_{p+q=r}\mathscr{D}'^{p,q}(X/\mathcal{F}).
\end{equation}
The space $\mathscr{D}'^{p,q}(X/\mathcal{F})$ is called the space of basic currents of bidimension $(n-p,n-q)$ and bidegree $(p,q)$ on $X$, and it is also denoted by $\mathscr{D}'_{n-p,n-q}(X/\mathcal{F})$.
It follows from Lemma \ref{tdglemma} that the basic
Dolbeault cohomology $H^{p,q}(X/\mathcal{F})$ of the foliation $\mathcal{F}$
is given by
\begin{equation}
H^{p,q}(X/\mathcal{F},\mathbb{C}):=\frac{\{\varphi\in\Omega^{p,q}
(X/\mathcal{F}):\bar\partial \varphi=0\}}
{\bar\partial  \Omega^{p,q-1}(X/\mathcal{F})}
\simeq\frac{\{T\in\mathscr{D}'^{p,q}(X/\mathcal{F}):\bar\partial  T=0\}}{\bar\partial  \mathscr{D}'^{p,q-1}(X/\mathcal{F})}.
\end{equation}
We also introduce the basic Bott-Chern cohomology group defined by
\begin{equation}
H_{\mathrm{BC}}^{p,q}(X/\mathcal{F},\mathbb{C}):=\frac{\{\varphi\in\Omega^{p,q}
(X/\mathcal{F}):\bar\partial \varphi=0\}}
{\partial\bar\partial  \Omega^{p-1,q-1}(X/\mathcal{F})}
\simeq\frac{\{T\in\mathscr{D}'^{p,q}(X/\mathcal{F}):\bar\partial  T=0\}}{\partial\bar\partial  \mathscr{D}'^{p-1,q-1}(X/\mathcal{F})}.
\end{equation}
It follows from \cite{elk90} that both $\dim_{\mathbb{R}}H^p(X/\mathcal{F})$ and $\dim_{\mathbb{C}}H_{\mathrm{BC}}^{p,q}(X/\mathcal{F})$ are finite (See \cite{elk2014} for more details). If $(X,\mathcal{F})$ is a transverse K\"ahler foliated manifold, then $H^{p,q}(X/\mathcal{F},\mathbb{C})=H_{\mathrm{BC}}^{p,q}(X/\mathcal{F},\mathbb{C})$
by the basic $\partial\bar\partial$-lemma \cite{baragliahekmati2018}.

Let $\mathcal{F}$ be a transverse holomorphic foliation on $X$ with real dimension $k$ and complex codimension $n$. Then for a holomorphic Hermitian $\mathcal{F}$-bundle $(E,h)$ with rank $r$, the adapted Chern connection $\nabla$ on $E$ is the unique basic connection which preserves $h$ and satisfies $\nabla^{(1,0)}=\bar\partial$.
 We denote by $c(E,h)$ the curvature of the connection $\nabla,$ which is a basic $(1,1)$ form with values in $\mathrm{End}(E).$ By the Chern-Weil theory \cite{chernclass}, the Chern form $c_j(E,h)$ of the   holomorphic Hermitian  $\mathcal{F}$-bundle $(E,h)$ is defined by
\begin{equation*}
\det\left(\mathrm{Id}_E+\frac{\sqrt{-1}}{2\pi}c(E,h)\right)=1+\sum_{i\geq 1}c_i(E,h),
\end{equation*}
where $c_i(E,h)$ is a closed basic real $(i,i)$ form for $i\geq 1.$ We also call $c_1(E,h)$ the Chern-Ricci form. We say that
$$
c_i^{\mathrm{BC}}(E/\mathcal{F})=[c_i(E,h)]\in H^{i,i}_{\mathrm{BC}}(X/\mathcal{F},\mathbb{R}):=\frac{\{\mathrm{d}\text{-closed basic real $(i,i)$ forms} \}}{\sqrt{-1}\partial \bar\partial \{\text{basic real $(i-1,i-1)$ forms}\}}
$$
is the $i^{\mathrm{th}}$ basic Chern class of $E.$ In particular, the first basic Chern class $c_1(X/\mathcal{F})$ of $(X,\mathcal{F})$ is defined by $-c_1(K_{X/\mathcal{F}}/\mathcal{F})$.

Let $\{s_1,\cdots,s_r\}$ be a local basic basis of transverse holomorphic Hermitian $\mathcal{F}$-bundle $E$ in the distinguished chart $(U;t^1,\cdots,t^k,z^1,\cdots,z^n)$. Then $h=(h_{\alpha\bar\beta})$ with $h_{\alpha\bar\beta}\in C^\infty(X/\mathcal{F},\mathbb{C})$ is the transverse Hermitian metric on $E$. Let
\begin{equation}
\label{defneicomplex}
e_i=\partial_i+\sum_{P=1}^kA_i^P\frac{\partial}{\partial t^P},\quad 1\leq i\leq n
\end{equation}
be a local basic basis of $\mathcal{\nu}^{1,0}$ with dual $\{\mathrm{d}z^1,\cdots,\mathrm{d}z^n\}$. Then we use the notation
$$
\nabla_i:=\nabla_{e_i},\quad \nabla_{\bar j}:=\nabla_{\bar e_j}
$$
and
$$
\nabla_is_\alpha=\sum_{\beta}\Gamma_{i\alpha}^\beta s_\beta,\quad \mbox{with}\quad \Gamma_{i\alpha}^\beta
=\sum_{\gamma}h^{\beta\bar\gamma}\partial_ih_{\alpha\bar\gamma},\quad \mbox{and}\quad
\sum_{\gamma}h^{\alpha\bar\gamma}h_{\beta\bar\gamma}=\delta^\alpha_\beta.
$$
Note that when $\nabla$ acts on basic sections, we have
$$
\nabla_i=\nabla_{\partial_i},\quad \nabla_{\bar j}=\nabla_{\partial_{\bar j}},\quad
\nabla_{[e_i,\bar e_j]}=\nabla_{[\bar e_i,\bar e_j]}=\nabla_{[e_i,e_j]}=0.
$$
In what follows, we will always study basic sections. Hence we will use this fact directly without  explanation.

The curvature is defined by
$$
\nabla_i\nabla_{\bar j}s_\alpha-\nabla_{\bar j}\nabla_{i}s_\alpha
-\nabla_{[e_i,\bar e_j]}s_\alpha=R_{i\bar j\alpha}{}^\beta s_\beta,\quad R_{i\bar j \alpha}{}^\beta:=-\partial_{\bar j}\Gamma_{i\alpha}^\beta.
$$
We use the notation $R_{i\bar j\alpha\bar\beta}:=R_{i\bar j\alpha}{}^\gamma h_{\gamma\bar\beta}$, and have
$$
R_{i\bar j\alpha\bar \beta}:=-\partial_i\partial_{\bar j}h_{\alpha\bar\beta}+h^{\gamma\bar\delta}\left(\partial_ih_{\alpha\bar\delta}\right)
\left(\partial_{\bar j}h_{\gamma\bar\beta}\right).
$$
The Chern-Ricci form $c_1(E,h)$ is given by
\begin{equation}
2\pi c_1(E,h)=\sqrt{-1}R_{i\bar j\alpha}{}^\alpha\mathrm{d}z^i\wedge\mathrm{d}\bar z^j=-\sqrt{-1}\partial\bar\partial\log\det(h_{\alpha\bar\beta}).
\end{equation}
For the transverse Hermitian metric $\omega=\sqrt{-1}g_{i\bar j}\mathrm{d}z^i\wedge\mathrm{d}\bar z^j$ with $g_{i\bar j}\in C^\infty(X/\mathcal{F},\mathbb{C})$,
the curvature tensor $R =\{R _{i\bar jk\bar \ell}\}$ of $\omega$ is given by (see \cite{elk90})
$$
R_{i\bar jk\bar \ell}=-\partial_i\partial_{\bar j}g_{k\bar\ell}+g^{p\bar q}\left(\partial_ig_{k\bar q}\right)\left(\partial_{\bar j}g_{p\bar \ell}\right).
$$
Given $\mathbf{x}\in X$ and basic vector field $W\in \nu^{1,0}_{\mathbf{x}}\setminus\{0\}$, the \emph{transverse holomorphic sectional curvature of $\omega$ at $\mathbf{x}$ in the direction $W$} is
$$H_{\mathbf{x}}(W):=\frac{R(W,\overline W,W,\overline W)}{|W|^4_\omega}.$$

We say that $\omega$ has negative transverse holomorphic sectional curvature if
$$
H_{\mathbf{x}}(W)<0,\quad \forall\;\mathbf{x}\in X,\;\forall\;W\in \nu^{1,0}_{\mathbf{x}}\setminus\{0\}
$$
We set
$$H^\omega_x:=\sup \{H^\omega_x(W)|W\in T^{1,0}_xX\setminus\{0\}\}$$
and
$$\mu_{\omega}:=\sup_{x\in X}H^\omega_x.$$

The Chern-Ricci form $\mathrm{Ric}(\omega)$ is defined by
\begin{equation}
\label{chernricciform}
\mathrm{Ric}(\omega)=-\sqrt{-1}\partial\bar\partial\log\det(g_{i\bar j})
\end{equation}
and there holds
$$
2\pi c_1(X/\mathcal{F})=[\mathrm{Ric}(\omega)]\in H^{1,1}_{\mathrm{BC}}(X/\mathcal{F},\mathbb{R}).
$$
We also introduce the transverse ``torsion" tensor as in Hermitian geometry.
\begin{equation}
T_{ij}^k:=g^{k\bar q}\left(\partial_i g_{j\bar q}-\partial_jg_{i\bar q}\right),\quad
T_{ij\bar\ell}:=T_{ij}^k g_{k\bar\ell }.
\end{equation}
For a basic $(1,0)$ form $a=a_{\ell}\mathrm{d} z^{\ell}$, define covariant derivative $\nabla_ia_{\ell}$ by
\begin{align*}
\nabla_{i}a_{\ell}:=\partial_{i}a_{\ell}-\Gamma_{i\ell}^pa_p.
\end{align*}
Then we can deduce
\begin{align}\label{commutate}
[\nabla_{i},\nabla_{\overline{j}}]a_{\ell}=-R_{i\overline{j}\ell}{}^pa_p,\quad [\nabla_{i},\nabla_{\overline{j}}]a_{\overline{m}}=R_{i\overline{j}}{}^{\overline{q}}{}_{\overline{m}}a_{\overline{q}},
\end{align}
where $R_{i\overline{j}}{}^{\overline{q}}{}_{\overline{m}}=R_{i\overline{j}p}{}^{\ell}g^{\overline{q}p}g_{\ell \overline{m}}$.

For each basic function $u\in C^{2}(X/\mathcal{F},\mathbb{R})$, one can infer
\begin{align}\label{ricciidentityu}
\nabla_{i}u=\partial_{i}u,\quad \nabla_{\overline{j}}u=\partial_{\overline{j}}u,\quad \nabla_{\overline{j}}\nabla_{i}u=\partial_{\overline{j}}\partial_{i}u,\quad [\nabla_{i},\nabla_j]u=-T_{ij}^p\nabla_pu.
\end{align}
A direct calculation yields that (cf.\cite{twcrelle} for Hermitian manifold)
\begin{equation}
\label{ricciidentity}
\nabla_{\overline{m}}\partial_{\bar j}\partial_pu =\nabla_{\overline{j}}\partial_{\bar m}\partial_pu-\overline{T_{mj}^q}u_{p\overline{q}},
\quad \nabla_{\ell}\partial_{\bar q}\partial_iu=\nabla_{i}\partial_{\bar q}\partial_\ell u-T_{\ell i}^pu_{p\overline{q}}
\end{equation}
and
\begin{align}
\label{ricciid2times}
\nabla_{\bar \ell}\nabla_k\nabla_{\bar j}\nabla_iu
=&\nabla_{\bar j}\nabla_i\nabla_{\bar \ell}\nabla_ku
+R_{k\bar\ell i}{}^p\nabla_{\bar j}\nabla_pu
-R_{i\bar j k}{}^p\nabla_{\bar\ell}\nabla_pu\\
&-T_{ki}^p\nabla_{\bar\ell}\nabla_{\bar j}\nabla_p u
-\overline{T_{\ell j}^q}\nabla_{k}\nabla_{\bar q}\nabla_i u
+T_{ki}^p\overline{T_{\ell j}^q}\nabla_{\bar q}\nabla_p u,\quad\forall\;u\in C^\infty(X/\mathcal{F},\mathbb{R}).\nonumber
\end{align}
\begin{lem}
\label{tlem1-3.30}
Let $\mathcal{F}$ be a transverse holomorphic foliation on $X$ with real dimension $k$ and complex codimension $n$.
If $f\in \mathscr{D}'^{p,0}(X/\mathcal{F})$ satisfies that $\bar\partial f\in\mathscr{E}^{p,1}(X/\mathcal{F}),$ then $f\in\mathscr{E}^{p,0}(X/\mathcal{F})$.
\end{lem}
\begin{proof}
The result is local, and hence we may assume that $X=V\times U\subset\mathbb{R}^k\times \mathbb{C}^n$ is a distinguish patch. Then
one infers from Proposition \ref{ramondexercise4.2.9} that $f=1\otimes \tilde f$ where $\tilde f\in \mathscr{D}'^{p,0}(U)$ satisfies $\bar\partial \tilde f\in \mathscr{E}^{p,1}(U)$, and hence it follows from \cite[Corollary
\uppercase\expandafter{\romannumeral1}-3.30]{demaillybook1} that $\tilde f\in \mathscr{E}^{p,0}(U)$, which means $f=1\otimes \tilde f\in \mathscr{E}^{p,0}((V\times U)/\mathcal{F}_{\upharpoonright V\times U}),$ as desired.
\end{proof}
Let $\nu$ be a Hermitian $\mathcal{F}$-bundle and $\omega$ denote the basic Hermitian metric on $\nu$. Then $\omega^n\wedge\chi$ is the canonical orientation on $X$ and $\omega^n$ is the canonical transverse orientation. In the distinguished chart $(U;t^1,\cdots,t^k,z^1,\cdots,z^n)$, each $u\in \Gamma(\bigwedge^{p,q})$ can be written as
 $$
u=\frac{1}{p!q!}u_{i_1,\cdots,i_p,\bar j_1,\cdots,\bar j_q}\mathrm{d}z^{i_1}\wedge\cdots\wedge\mathrm{d}z^{i_p}\wedge\mathrm{d}\bar z^{j_1}\wedge\cdots\wedge\mathrm{d}\bar z^{j_q},
 $$
 where $u_{i_1,\cdots,i_p,\bar j_1,\cdots,\bar j_q}=u_{i_1,\cdots,i_p,\bar j_1,\cdots,\bar j_q}(t,z)$. 
 
However, given Lemma \ref{lembasicfunctionwithcompactsupport0}, we should introduce the notion of positivity similar to the one in \cite{lelong57} carefully (c.f. \cite[Remark 1.15 of Section 1 of Chapter 3]{demaillybook1} and \cite{vancoevering2016}).
 
A form $u\in \Gamma(\bigwedge^{p,p})$ is called positive (resp. strictly positive) if on  each distinguished chart $(U;t^1,\cdots,t^k,z^1,\cdots,z^n)$
$$
\frac{u_{\upharpoonright U}\wedge\sqrt{-1}\alpha_1\wedge\overline{\alpha}_1
\wedge\cdots\wedge
\sqrt{-1}\alpha_{n-p}\wedge\overline{\alpha}_{n-p}
\wedge\chi}{\left(\omega^n\wedge\chi\right)_{\upharpoonright U}}\geq 0\;(\mbox{resp.}\;>0)
$$
on $U$ for each non-zero basic $\alpha_j\in \Gamma(\nu^{1,0}_{\upharpoonright U})$ with constant coefficients under the coordinate $(z^1,\cdots,z^n)$ and $1\leq j\leq n-p$.

In order to introduce the notion of positivity of basic $(p,p)$ current, let us recall the notion of invariant transverse measure.
Let $\mathcal{F}$ be a transverse holomorphic foliation on $X$ with real dimension $k$ and complex codimension $n$. Then an invariant transverse measure $\mu$ is a measure on the local leaf space in each chart which is preserved by transition
functions. We can identify the invariant transverse measure $\mu$ with the basic $(n,n)$ current $\mathcal{C}_\mu$  given by
$$
\mathcal{C}_\mu:=(\sqrt{-1})^{n}\sum \left(1\otimes \mu\right)\mathrm{d}z^1\wedge\mathrm{d}\bar z^1\wedge\cdots\wedge\mathrm{d}z^n\wedge\mathrm{d}\bar z^n
$$
in the distinguished chart $(t^1,\cdots,t^k,z^1,\cdots,z^n)$. In what follows, we will not distinguish $\mu,\;\mathcal{C}_u$ and $1\otimes \mu$.

A basic $(p,p)$ current $T$ is called positive if on  each distinguished chart $(U;t^1,\cdots,t^k,z^1,\cdots,z^n)$
$$
T_{\upharpoonright U}\wedge \sqrt{-1}\alpha_1\wedge\bar\alpha_1\wedge\cdots\wedge\sqrt{-1}
\alpha_{n-p}\wedge\bar\alpha_{n-p}
$$
is a  positive invariant transverse measure on $U$  for each non-zero basic $\alpha_j\in \Gamma(\nu^{1,0}_{\upharpoonright U})$ with constant coefficients in the given coordinates $(z^1,\cdots,z^n)$ and $1\leq j\leq n-p$.  The set of positive basic currents of bi-dimension $(n-p, n-p)$ will
be denoted by
$$
\mathscr{D}'^{+}_{n-p,n-p}(X/\mathcal{F}).
$$
In the distinguished chart $(U;t^1,\cdots,t^k,z^1,\cdots,z^k)$,
it follows from \cite[Lemma III-1.4]{demaillybook1} that $\bigwedge^{p,p}$ admits a basis consisting of
\label{lem3-1.4}
$$
\alpha_s=\sqrt{-1}\alpha_{s,1}\wedge\overline{\alpha}_{s,1}\wedge\cdots\wedge \sqrt{-1}\alpha_{s,p}\wedge\overline{\alpha}_{s,p},\quad 1\leq s\leq \binom{n}{p} ^2
$$
where $\alpha_s$ is of the type $\mathrm{d} z^j\pm\mathrm{d} z^k$ or $\mathrm{d} z^j\pm\sqrt{-1} \mathrm{d} z^k,\,1\leq j,k\leq n$. This, together with the argument in \cite[Proposition III-1.14]{demaillybook1}, yields that
 each positive basic current $$T =\frac{(\sqrt{-1})^{p^2} }{p!p!}T_{i_1,\cdots,i_p,\bar j_1,\cdots,\bar j_p}\mathrm{d}z^{i_1}\wedge\cdots\wedge\mathrm{d}z^{i_p}\wedge\mathrm{d}\bar z^{j_1}\wedge\cdots\wedge\mathrm{d}\bar z^{j_p}\in \mathscr{D}'^{+}_{n-p,n-p}(X/\mathcal{F})$$
 is real and of order $0$, i.e., its coefficients $T_{I,J}$ are invariant transverse complex measures and satisfy $$\overline{T_{i_1,\cdots,i_p,\bar j_1,\cdots,\bar j_p}}=T_{j_1,\cdots,j_p,\bar i_1,\cdots,\bar i_p}.$$
Moreover  $T_{i_1,\cdots,i_p,\bar i_1,\cdots,\bar i_p}$ is positive  invariant transverse measure, and the total variation  $|T_{i_1,\cdots,i_p,\bar j_1,\cdots,\bar j_p}|$
of the measures $T_{i_1,\cdots,i_p,\bar j_1,\cdots,\bar j_p}$ satisfy the inequality
\begin{align*}
\lambda_{i_1}\cdots\lambda_{i_p}\lambda_{j_1}\cdots,\lambda_{j_p}|T_{i_1,\cdots,i_p,\bar j_1,\cdots,\bar j_p}|\leq 2^p\sum_{m_1,\cdots,m_p}(\lambda_{m_1}\cdots\lambda_{m_p})^2T_{m_1,\cdots,m_p,\bar m_1,\cdots,\bar m_p},
\end{align*}
where $\lambda_k\geq0$ and
 $$
 \{i_1,\cdots,i_p\}\cap \{j_1,\cdots,j_p\}
 \subset\{m_1,\cdots,m_p\}\subset \{i_1,\cdots,i_p\}\cup \{j_1,\cdots,j_p\}.
 $$
 We denote by ${}^s\mathscr{D}^{p,q}(X,\mathcal{F})$ the set of $(k+p+q)$-forms given locally by
 $$
 \frac{1}{p!q!}\varphi_{i_1,\cdots,i_p,\bar j_1,\cdots,\bar j_q}\theta^1\wedge\cdots\wedge\theta^k\wedge\mathrm{d}z^{i_1}\wedge\cdots\wedge\mathrm{d}z^{i_p}\wedge\mathrm{d}\bar z^{j_1}\wedge\cdots\wedge\mathrm{d}\bar z^{j_q}
 $$
 with $\varphi_{i_1,\cdots,i_p,\bar j_1,\cdots,\bar j_q}=\varphi_{i_1,\cdots,i_p,\bar j_1,\cdots,\bar j_q}(t,z)\in C^s(U,\mathbb{R})$ and $\{\theta^1,\cdots,\theta^k\}$ defined similarly  by \eqref{defnthetaalpha}.

The dual of ${}^s\mathscr{D}^{p,q}(X,\mathcal{F})$ is denoted by ${}^s\mathscr{D}'^{n-p,n-q}(X,\mathcal{F})$ or ${}^s\mathscr{D}'_{p,q}(X,\mathcal{F})$.
 With these notations, we get that $\mathscr{D}'^{+}_{n-p,n-p}(X/\mathcal{F})$ is weakly-* closed cone in  ${}^0\mathscr{D}'_{p,p}(X,\mathcal{F})$ (see \cite{craioveanuputa1980,derham55}). Note that ${}^0\mathscr{D}^{p,p}(X,\mathcal{F})$ is a separable topological vector space.
\begin{lem}
\label{tprop3-1.23}
Let $\mathcal{F}$ be a transverse holomorphic foliation on $X$ with real dimension $k$ and complex codimension $n$, and let $\delta$ be a positive continuous function on $X$. Then the set given by
$$
\mathscr{T}:=\left\{T\in \mathscr{D}'^+_{p,p}(X/\mathcal{F}):\;\int_X\delta T\wedge\omega^p\wedge\chi\leq 1\right\}
$$
is weakly-* compact and weakly-* sequently compact.
\end{lem}
\begin{proof}
This is a transverse version of \cite[Proposition III-1.23]{demaillybook1}. Since
$$
V:=\{u\in {}^0\mathscr{D}^{p,p}(X,\mathcal{F}):\;-\delta \omega^p\wedge\chi<u<\delta\omega^p\wedge\chi\}
$$
is an open neighborhood of $0$ in the separable topological vector space ${}^0\mathscr{D}^{p,p}(X,\mathcal{F})$ (actually we can check this directly by using the topology of ${}^0\mathscr{D}^{p,p}(X,\mathcal{F})$), we get that
$$
\tilde{\mathscr{T}}:=\left\{T\in {}^0\mathscr{D}'_{p,p}(X,\mathcal{F}):\;|\langle T,u\rangle|\leq 1,\quad \forall\;u\in V \right\}
$$
is weakly-* compact and weakly-* sequently compact
from
\cite[Theorem 3.64, Theorem 3.67 \& Corollary 3.68]{kupfermantvs}.
Hence Lemma \ref{tprop3-1.23} follows from the fact that $\mathscr{T}$ is a weakly-* closed subset contained in $\tilde{\mathscr{T}}$.
\end{proof}

\subsection{Notions of positivity}
Let $[\alpha]\in H_{\mathrm{BC}}^{1,1}(X/\mathcal{F},\mathbb{R})$, where $\alpha$ is a smooth closed real $(1,1)$ basic form. We define the following positivity notions which are easily seen to be independent of the choice of $\omega$ and hence without loss of generality we assume that $\omega$ is the basic Gauduchon metric (i.e., $\omega$ is a positive $(1,1)$ basic form such that $\partial  \bar\partial  \omega^{n-1}=0$). It follows from \cite{baragliahekmati2018} that there exists a unique basic Gauduchon metric up to scaling (when $n\geq 2$) in  the conformal class of each Hermitian metric.
\begin{itemize}
\item $[\alpha]$ is {\bf transverse K\"ahler} if it contains a representative which is a basic K\"ahler form, i.e., if there is a smooth basic function $\varphi$ such that $\alpha+\sqrt{-1}\partial\overline{\partial}\varphi\geq\varepsilon\omega$ on $X$, for some $\varepsilon>0$.
\item $[\alpha]$ is {\bf transverse nef} if for every $\varepsilon>0$ there is a smooth basic function $\varphi_\varepsilon$ such that
\begin{equation}
\label{defntnef}
\alpha+\sqrt{-1}\partial\overline{\partial}\varphi_\varepsilon\geq-\varepsilon\omega
\end{equation}
holds on $X$.
\item $[\alpha]$ is {\bf transverse big} if it contains a basic K\"ahler current $T$, i.e., there exists a closed basic current $T\in[\alpha]$ such that $T\geq\varepsilon\omega$ holds weakly as currents on $X$  for some $\varepsilon>0$.
\item $[\alpha]$ is {\bf transverse pseudoeffective} if it contains a closed positive basic current.
\end{itemize}
The set of all the transverse K\"ahler classes (resp. transverse pseudoeffective classes, transverse nef classes, transverse big classes) is denoted by   $\mathcal{K}_{X/\mathcal{F}}$ (resp. $\mathcal{E}_{X/\mathcal{F}}$, $\mathcal{N}_{X/\mathcal{F}}$,  $\mathcal{B}_{X/\mathcal{F}}$).
\begin{lem}
Let $\mathcal{F}$ be a transverse holomorphic foliation on $X$ with real dimension $k$ and complex codimension $n$. Then both $\mathcal{E}_{X/\mathcal{F}}$ and $\mathcal{N}_{X/\mathcal{F}}$ are closed cone. Moreover, there holds that
$\mathcal{N}_{X/\mathcal{F}}\subset\mathcal{E}_{X/\mathcal{F}}$ and
$\mathcal{E}_{X/\mathcal{F}}\cap(-\mathcal{E}_{X/\mathcal{F}})=\{0\}$.
\end{lem}
\begin{proof}
We just check the closeness by adapting the idea of \cite[Proposition 6.1]{dem92b}. Without out loss of generality, we assume that $\omega$ is a transverse Gauduchon metric. For the closeness of  $\mathcal{E}_{X/\mathcal{F}}$, we assume that $\{T_i\}_{i=1}^\infty$ is a sequence of closed positive basic currents such that
$$
[T_i]\to [\alpha],\quad \mbox{as}\quad i\to \infty.
$$
Then it follows from \eqref{bst} that
\begin{equation*}
\int_MT_i\wedge\omega^{n-1}\wedge\chi=\int_M [T_i]\wedge \omega^{n-1}\wedge\chi
\to \int_M[\alpha]\wedge\omega^{n-1}\wedge\chi,\quad \mbox{as}\quad i\to \infty.
\end{equation*}
This, together with Lemma \ref{tprop3-1.23}, yields that there exists a subsequence $\{T_{i_j}\}_{j=1}^\infty$ converging to a closed positive basic current $T$ weakly. Hence we can deduce $[T_{i_j}]\to [T]$ as $j\to \infty$ and that $[T]=[\alpha]$ by the uniqueness of weak limit is transverse pseudoeffective, as required.

For the closeness of  $\mathcal{N}_{X/\mathcal{F}}$, we assume that $\{\alpha_i\}_{i=1}^\infty$ is a sequence of closed smooth real basic $(1,1)$ form such that
$$
[\alpha_i]\to [\alpha],\quad \mbox{as}\quad i\to \infty.
$$
We select a sequence of smooth representatives $\gamma_j\in[\alpha]-[\alpha_j]$ which converges to $0$ in $C^\infty(X/\mathcal{F},\mathbb{R})$. Fix $\varepsilon>0$. If $[\alpha_j]$ is transverse nef, then there exists a smooth representative $\alpha_{j,\varepsilon}$ such that $\alpha_{j,\varepsilon}>-\frac{\varepsilon}{2} \omega$. On the other hand, there exists a $k_0\in \mathbb{N}$ such that $\gamma_j> -\frac{\varepsilon}{2}\omega$ with $j\geq k_0$. Thus, one can infer that $[\alpha]=[\alpha_j]+[\gamma_j]$ contains a representative $\alpha_{j,\varepsilon}+\gamma_j$
with
$$
\alpha_{j,\varepsilon}+\gamma_j>-\varepsilon\omega,\quad \forall\; j\geq k_0.
$$
This yields that $[\alpha]$ is transverse nef and hence that $\mathcal{N}_{X/\mathcal{F}}$ is closed.

Let $[\alpha]$ is a transverse nef class. Then for each $\varepsilon>0$, there exists a smooth basic function $\varphi_\varepsilon\in C^\infty(X/\mathcal{F},\mathbb{R})$ such that
$$
\alpha+\sqrt{-1}\partial\bar\partial\varphi_{\varepsilon}>-\varepsilon\omega.
$$
This yields that
\begin{equation*}
-\varepsilon\int_X\omega^n\wedge\chi\leq \int_X(\alpha+\sqrt{-1}\partial\bar\partial\varphi_\varepsilon )\wedge \omega^{n-1}\wedge\chi=\int_X \alpha \wedge\omega^{n-1}\wedge\chi.
\end{equation*}
This, together with Lemma \ref{tprop3-1.23}, yields that we can assume that without loss of generality $\lim_{\varepsilon\to 0}(\alpha+\sqrt{-1}\partial\bar\partial\varphi_\varepsilon)=T$ in the current sense, where $T$ is a closed positive basic current. Hence $[\alpha]=[T]$  is transverse pseudoeffective.

If $T$ is a closed positive basic current on $X$ and the class $-[T]$ is also transverse pseudoeffective, then there is a closed positive basic current $\tilde{T}=-T+\sqrt{-1}\partial\bar\partial\varphi$ for some basic distribution $\varphi$, and so $\sqrt{-1}\partial\bar\partial\varphi=T+\tilde{T}>0$  in the current sense. It follows from \cite[Proposition 1.43]{bookgz17} and Proposition \ref{ramondexercise4.2.9} that we can see $\varphi$ as a basic plurisubharmonic function and hence $\varphi$ is a constant by the maximum principle. This yields that both $T$ and $\tilde{T}=-T$ are positive in the current sense and hence $T=0$, as desired.
\end{proof}
\begin{lem}
Let $\mathcal{F}$ be a transverse K\"ahler foliation on $X$ with real dimension $k$ and complex codimension $n$. Then the the transverse K\"ahler cone $\mathcal{K}_{X/\mathcal{F}}$ is an open and convex cone inside $H^{1,1}(X/\mathcal{F},\mathbb{R})$. Furthermore, there holds that $\mathcal{K}_{X/\mathcal{F}}\cap(-\mathcal{K}_{X/\mathcal{F}})=\{0\}$.
\end{lem}
\begin{proof}
This is a transverse version of \cite{tosattikawa} for the openness of K\"ahler cone and we use the idea adapted from \cite{tosattikawa}. By saying that $\mathcal{K}_{X/\mathcal{F}}$ is a cone we mean that if we are given $[\alpha]\in\mathcal{K}_{X/\mathcal{F}}$ and $\lambda\in\mathbb{R}_{>0}$, then $\lambda[\alpha]\in\mathcal{K}_{X/\mathcal{F}}$, which is obvious. The convexity of $\mathcal{K}_{X/\mathcal{F}}$ follows immediately from the fact that if $\omega_1$ and $\omega_2$ are transverse K\"ahler metrics on $(X,\mathcal{F})$ and $0\leq\lambda\leq 1$, then $\lambda\omega_1+(1-\lambda)\omega_2$ is also a tansverse K\"ahler metric. For the openness of $\mathcal{K}_{X/\mathcal{F}}$, we fix closed real basic $(1,1)$ forms $\{\alpha_1,\dots,\alpha_k\}$ on $(X,\mathcal{F})$ such that $\{[\alpha_1],\dots,[\alpha_k]\}$ is a basis of $H^{1,1}(X/\mathcal{F},\mathbb{R})$. Given a transverse K\"ahler class $[\alpha]\in\mathcal{K}_{X/\mathcal{F}}$ we can write
$[\alpha]=\sum_{i=1}^k\lambda_i[\alpha_i],$ for some $\lambda_i\in\mathbb{R}$. Since $[\alpha]\in\mathcal{K}_{X/\mathcal{F}}$, there exists a basic function $\varphi\in C^\infty(X/\mathcal{F},\mathbb{R})$ such that
$$\sum_{i=1}^k\lambda_i\alpha_i+\sqrt{-1}\partial\bar\partial\varphi>0.$$
Since $X$ is compact, it follows that
$$\sum_{i=1}^k\tilde{\lambda}_i\alpha_i+\sqrt{-1}\partial\bar\partial\varphi>0,$$
for all $\tilde{\lambda}_i$ sufficiently close to $\lambda_i$ ($1\leq i\leq k$), and so all $(1,1)$ classes in a neighborhood of $[\alpha]$ contain a transverse K\"ahler metric.

If $\omega$ is a transverse K\"ahler metric on $X$ and the class $-[\omega]$ is also transverse K\"ahler, then there is a transverse K\"ahler metric $\tilde{\omega}=-\omega+\sqrt{-1}\partial\bar\partial\varphi$ for some basic function $\varphi\in C^\infty(X,\mathcal{F})$, and so $\sqrt{-1}\partial\bar\partial\varphi=\omega+\tilde{\omega}>0$  everywhere on $X$. This is impossible, since $\sqrt{-1}\partial\bar\partial\varphi\leq 0$ at the points where $\varphi$ attains its maximum. Hence we have $\mathcal{K}_{X/\mathcal{F}}\cap(-\mathcal{K}_{X/\mathcal{F}})=\emptyset$.
\end{proof}
\begin{lem}\label{nef2}
Let $\mathcal{F}$ be a transverse K\"ahler foliation on $X$ with real dimension $k$ and complex codimension $n$. Then there holds $\mathcal{E}_{X/\mathcal{F}}=\overline{\mathcal{K}_{X/\mathcal{F}}}$ and
 $\mathcal{E}_{X/\mathcal{F}}^{\circ}=\mathcal{B}_{X/\mathcal{F}}$.
\end{lem}
\begin{proof}
This is a transverse version of \cite[Lemma 2.2]{tosattikawa}. We use the idea adapted from \cite{tosattikawa} and assume that  $\omega$ is a transverse K\"ahler metric.

Condition in definition of transverse nef \eqref{defntnef} is equivalent to $[\alpha+\varepsilon\omega]\in\mathcal{K}_{X/\mathcal{F}}$, for all $\varepsilon>0$, which certainly implies that $[\alpha]\in \overline{\mathcal{K}_{X/\mathcal{F}}}$. Conversely, if
$[\alpha]\in \overline{\mathcal{K}_{X/\mathcal{F}}}$, then there is a sequence $\{\beta_i\}$ of closed real $(1,1)$ forms such that $\alpha+\beta_i>0$ for all $i$, and $[\beta_i]\to 0$ in $H^{1,1}(X/\mathcal{F},\mathbb{R})$.
As before we fix closed real $(1,1)$ forms $\{\alpha_1,\dots,\alpha_k\}$ on $X$ such that $\{[\alpha_1],\dots,[\alpha_k]\}$ is a basis of $H^{1,1}(X/\mathcal{F},\mathbb{R})$, and for each $i$ write
$$[\beta_i]=\sum_{j=1}^k \lambda_{ij}[\alpha_j],$$
with $\lambda_{ij}\in\mathbb{R}$. Since $[\beta_i]\to 0$, and $\{[\alpha_1],\dots,[\alpha_k]\}$ is a basis, we conclude that $\lambda_{ij}\to\infty$ as $i\to\infty$, for each fixed $j$.
If we let
$$\tilde{\beta}_i=\sum_{j=1}^k \lambda_{ij}\alpha_j,$$
then the forms $\tilde{\beta}_i$ converge smoothly to zero, as $i\to\infty$, and we can find smooth basic functions $\varphi_i\in C^\infty(X/\mathcal{F},\mathbb{R})$ such that
$\beta_i=\tilde{\beta}_i+\sqrt{-1}\partial\bar\partial\varphi_i$. For every $\varepsilon>0$ we choose $i$ sufficiently large so that $\tilde{\beta}_i<\varepsilon\omega$ on $X$, and so
$$\alpha+\varepsilon\omega+\sqrt{-1}\partial\bar\partial\varphi_i>\alpha
+\tilde{\beta}_i+\sqrt{-1}\partial\bar\partial\varphi_i=\alpha+\beta_i>0,$$
which proves \eqref{defntnef}.

The fact that transverse big class contains in $\mathcal{E}_{X/\mathcal{F}}^{\circ}$ is obvious. On the other hand, for each $[\alpha]\in \mathcal{E}_{X/\mathcal{F}},$ we have
\begin{equation*}
[\alpha]=\lim_{k\to \infty}\left([\alpha]+\frac{1}{k}[\omega]\right),
\end{equation*}
with $[\omega]\in \mathcal{K}_{X/\mathcal{F}}$. Since $\left([\alpha]+\frac{1}{k}[\omega]\right)$ is transverse big, it follows from the closeness of $\mathcal{E}_{X/\mathcal{F}}$ that
\begin{equation*}
\mathcal{E}_{X/\mathcal{F}}\subset \mbox{closure of transverse big cone} \subset \mathcal{E}_{X/\mathcal{F}},
\end{equation*}
which yields $\mathcal{E}_{X/\mathcal{F}}^{\circ}=\mathcal{B}_{X/\mathcal{F}}$.
\end{proof}
Clearly every transverse K\"ahler class is transverse nef and transverse big, and every transverse big class is transverse pseudoeffective. Aslo, a transverse nef class is transverse pseudoeffective.

There are in general no other implications among these notions since we can see $H_{\mathrm{BC}}^{1,1}(M,\mathbb{R})$ as $H_{\mathrm{BC}}^{1,1}(X/\mathcal{F},\mathbb{R})$, where $M$ is a complex manifold and  $\mathcal{F}$  is defined by a holomorphic submersion $\pi:\;X\to M$ (cf. \cite{tosattinakamaye}).


\section{Proof of the Main Theorem}
\label{section:mainthm}
\begin{lem}
\label{taubinyau}
Let $(X,\mathcal{F})$ be a closed oriented$,$ taut$,$ transverse Hermitian foliated manifold$,$ where $\mathcal{F}$ is the foliation with complex codimension $n,$ and $\omega$ denote a transverse Hermitian metric.
Given a smooth basic function $h\in C^\infty(X/\mathcal{F},\mathbb{R}),$ there exists a unique smooth basic function $u\in C^\infty(X/\mathcal{F},\mathbb{R})$ such that
\begin{equation}
(\omega+\sqrt{-1}\partial\bar\partial u)^n=e^{u+h}\omega^n,\quad
\omega+\sqrt{-1}\partial\bar\partial u>0.
\end{equation}
\end{lem}
\begin{proof}
This is a transverse version of \cite{aubin78,yau1978}. We will adapt the idea from \cite{gaborjdg,stw1503,twarxiv1906} (cf.\cite{fengzhengsasaki1}) to prove it.
We use the method of continuity and consider a family of equations
\begin{equation}\label{fvarphit}
\mathscr{H}(u,t):=\log\frac{(\omega+\sqrt{-1}\partial\bar\partial u)^n}{\omega^n}
-u-th=0,\quad \omega+\sqrt{-1}\partial\bar\partial u>0.
\end{equation}
We set
$$
\mathscr{T}:=\Big\{t\in[0,1]:\;\eqref{fvarphit}\mbox{ has a solution }u\in C^{2,\alpha}(X/\mathcal{F},\mathbb{R})\mbox{ with $\alpha\in(0,1)$ fixed}\Big\}.
$$
Note that $0\in\mathscr{T}$ since $\mathscr{H}(0,0)=0$. For the openness of $\mathscr{T}$, we fix $t\in \mathscr{T}$. A direct calculation yields that
\begin{equation*}
\left(D_u \mathscr{H}\right)_{(u_t,t)}\eta=\Delta_{\tilde{\omega}} \eta-\eta,\quad \forall\,\eta\in C^{2,\alpha}(X/\mathcal{F},\mathbb{R}),
\end{equation*}
where
$$
 \Delta_{\tilde{\omega}} \eta:=\frac{n\sqrt{-1}\partial\bar\partial\eta\wedge\tilde{\omega}^{n-1}}
{\tilde{\omega}^n},\quad \tilde{\omega} :=\omega+\sqrt{-1}\partial\bar\partial u_t>0.
$$
We claim that $(D_{u}\mathscr{H})_{(u_t,t)}$ is bijective. Indeed, since $\tilde{\omega}$ is a positive basic $(1,1)$ form, both $\tilde{\Delta}$ and $(D_{u}\mathscr{H})_{(u_t,t)}$ are strictly transverse elliptic operators. For each $\eta \in \ker (D_{u}\mathscr{H})_{(u_t,t)},$ at the point $\mathbf{x}_{\mathrm{max}}$ (resp. $\mathbf{x}_{\mathrm{min}}$) where $\eta$ attains its maximum (resp. its minimum), there holds
\begin{equation*}
-\eta(\mathbf{x}_{\mathrm{max}})\geq0 \;(\mbox{resp.}\;-\eta(\mathbf{x}_{\mathrm{min}})\leq0),
\end{equation*}
which yields that $\eta\equiv0$. Hence $\ker (D_{u}\mathscr{H})_{(u_t,t)}=\{0\}$, i.e., $(D_{u}\mathscr{H})_{(u_t,t)}$ is injective.

Since the index of $(D_{u}\mathscr{H})_{(u_t,t)}$ is the same as the index of the transverse Laplacian defined by $\tilde{\omega}$ which is zero by \cite[Lemma 3.5]{baragliahekmati2018}, the injectivity of the operator directly implies that the operator $(D_{u}\mathscr{H})_{(u_t,t)}$ is surjective.

Since $(D_{u}\mathscr{H})_{(u_t,t)}$ is bijective, the implicit function theorem yields that $\mathscr{T}$ is open at the point $t\in\mathscr{T}$.

For the closeness of $\mathscr{T}$, we need the a priori estimates for the solutions to \eqref{fvarphit}.

\textbf{A uniform bound }
\begin{equation}
\label{equzeroestimate}\sup_X|u|  \leq C
\end{equation}
follows immediately from the standard maximum principle. Indeed, at the point $\mathbf{x}_{\mathrm{max}}$ where $u$ attains its maximum, we have (see \cite{twjams})
$$
\left(\sqrt{-1}\partial\bar\partial u\right)(\mathbf{x}_{\mathrm{max}})\leq0,
$$
and from \eqref{fvarphit} the upper bound of $u\leq -t\inf_Xh$ follows. The lower bound of $u$ is similar.
Here and henceforth, $C$ will denote a uniform constant independent of $t$ that may change from line to line.

We denote $F:=u+th$. We assume that $\sup_Xu\leq 0$ by modifying $\tilde{u}:=u+t\inf_Xh$ and
$\tilde{h}=h-\inf_Xh$ such that $F=\tilde{u}+t\tilde{h} $.

\textbf{For the second order estimate}, we claim
\begin{equation}
\label{equ2ndestimate}
\sup_M|\partial\bar\partial u|_\omega\leq CK,
\end{equation}
with $K:=1+\sup_M|\partial u|_\omega^2$.

In the distinguished chart $(U;t^1,\cdots,t^k,z^1,\cdots,z^n),$
we write
\begin{equation}
\label{defnvartheta}
\tilde\omega:=\omega+\sqrt{-1}\partial\bar\partial u_t=\sqrt{-1}\tilde{g}_{i\bar j}\mathrm{d}z^i\wedge\mathrm{d}\bar z^j,\quad \omega=\sqrt{-1}g_{i\bar j}\mathrm{d}z^i\wedge\mathrm{d}\bar z^j.
\end{equation}
It follows from \eqref{fvarphit} and the inequality of arithmeticand geometric mean that
\begin{equation}
\label{ftau}
\mathrm{tr}_{\tilde\omega}\omega\geq n\left(\inf_X\frac{\omega^n}{\tilde\omega^n}\right)^{\frac{1}{n}}\geq \tau>0,
\mathcal{}\end{equation}
where $\tau$ is a constant independent of $t$.

Let $\lambda_1 \ge \lambda_2 \ge \cdots \ge \lambda_n>0$ be the eigenvalues of $\tilde{\omega}$ with respect to $\omega$.  We consider the quantity
$$
H(\mathbf{x}):= \log \lambda_1(\mathbf{x})+\varphi(|\partial u|_g^2(\mathbf{x}))+ \psi(u(\mathbf{x})),\quad \forall\,\mathbf{x}\in M,
$$
where we define
$$\varphi(t) = \frac{t}{K}, \; t\ge 0, \quad \textrm{and} \quad \psi(t) =  e^{-At}, \; t \le 0,$$
with
$$K = \sup_X | \partial u|^2_g+1,$$
and $A>0$ to be determined later.
A direct calculation yields that
$$ - \psi' \ge A>0, \quad \psi'' = -A\psi'.$$
We assume that $H$ attains its maximum at the point $\mathbf{x}_0\in X.$
In the following, we will calculate at the point $\mathbf{x}_0$ under the distinguished  coordinate $(t^1,\cdots,t^k,z^1,\cdots,z^n)$ for which $\omega$ is the identity and $(\tilde{g}_{i\bar j})$ is diagonal with entries $\tilde g_{i\bar i}=\lambda_i$ for $1\leq i\leq n$, unless otherwise indicated.

Since $\lambda_1$ may not be smooth at $\mathbf{x}_0,$ we introduce a smooth function $\phi$ on $M$ by (cf. \cite[Lemma 5]{brendleetx2017} and \cite[Proof of Theorem 3.1]{twarxiv1906})
 \begin{equation}
 \label{defnphi}
H(\mathbf{x}_0)\equiv \log f(\mathbf{x}) +\varphi(|\partial u|_g^2(\mathbf{x}))+\psi(u(\mathbf{x})),\quad \forall\,\mathbf{x}\in M.
\end{equation}
Note that $f$ satisfies
\begin{equation}
f(\mathbf{x})\geq \lambda_1(\mathbf{x})\quad \forall\,\mathbf{x}\in X,\quad f(\mathbf{x}_0)=\lambda_1(\mathbf{x}_0).
\end{equation}
Note that
\begin{equation}
\label{defnl}
\Delta_{\tilde{\omega}} u=\tilde{g}^{i\bar j}\partial_i\partial_{\bar j}u=n-\mathrm{tr}_{\tilde \omega}\omega.
\end{equation}
Applying the operator $\Delta_{\tilde{\omega}}$ defined in \eqref{defnl} to \eqref{defnphi}, one infers
\begin{equation}
\label{0lhatx0}
0=\frac{1}{\lambda_1}\Delta_{\tilde{\omega}}f-\frac{1}{\lambda_1^2}\tilde{g}^{i\bar i}|\nabla_if|^2
+\varphi'\Delta_{\tilde\omega}(|\partial u|_{\omega}^2)
+\varphi''\tilde{g}^{i\bar i}\left|\partial_i|\partial u|_{\omega}^2\right|^2
+\psi'\Delta_{\tilde{\omega}}u+\psi''\tilde{g}^{i\bar i}|\partial_iu|^2.
\end{equation}
Differentiating \eqref{defnphi} one can deduce
\begin{equation}
\label{nalbaiphi}
0=\frac{\nabla_if}{f}+\varphi'\nabla_i(|\partial u|_{\omega}^2) +\psi'(\nabla_iu).
\end{equation}
\begin{lem}
\label{lemddbarphi}
Let $\mu$ denote the multiplicity of the largest eigenvalue of $(\tilde{g}_{i\bar j})$ with respect to $(g_{i\bar j})$ at $\mathbf{x}_0,$ so that
$\lambda_1=\cdots=\lambda_\mu>\lambda_{\mu+1}\geq \cdots\geq \lambda_n.$ Then at $\mathbf{x}_0,$ for each $i$ with $1\leq i\leq n,$ there hold
\begin{align}
\label{nablaiphi}
&\nabla_i\tilde{g}_{k\bar\ell}= (\nabla_if) g_{k\bar \ell},\quad \mbox{for}\quad 1\leq k,\,\ell\leq \mu,\\
\label{ddbarphi}
&\nabla_{\bar i}\nabla_if \geq \nabla_{\bar i}\nabla_{i}\tilde{g}_{1\bar 1}+\sum_{q>\mu}\frac{\left|\nabla_i\tilde{g}_{q\bar 1}\right|^2+\left|\nabla_{\bar i}\tilde{g}_{q\bar 1}\right|^2}{\lambda_1-\lambda_q}.
\end{align}
\end{lem}
\begin{proof}
See \cite[Lemma 3.2]{twarxiv1906}.
\end{proof}
Since we work with $(z^1,\cdots,z^n)$ at the point $\mathbf{x}_0$ under the distinguished  coordinate $$(t^1,\cdots,t^k,z^1,\cdots,z^n),$$ the conclusion follows from the same argument of \cite[Proof of Theorem 3.1]{twarxiv1906} with vanishing gradient term by replacing $u_{ij}-u_{ji}=T_{ji}^pu_p$, the linearized operator $L$ in \cite[Formula (3.1)]{twarxiv1906},
\cite[Formula (3.13)]{twarxiv1906} and \cite[Formula (3.14)]{twarxiv1906}
with \eqref{ricciidentityu}, $\Delta_{\tilde\omega}$, \eqref{ricciid2times} and \eqref{ricciidentity} respectively. We also point out that here $F=u+th$ contains the solution $u$ which is harmless. Indeed, it follows from   \eqref{ftau} that
$$
\nabla_1\nabla_{\bar 1}F=O(\lambda_1)
$$
will be absorbed into $\lambda_1\left(\mathrm{tr}_{\tilde{\omega}}\omega\right)$, and that, together with the definition of $\varphi,$ we know that
$$
\varphi' (\partial_pF)(\partial_{\bar p}u)
=O(1)
$$
is harmless.

\textbf{The first order estimate}
\begin{equation}
\label{equ1stestimate}
\sup_M|\partial  u|_\omega\leq C
\end{equation}
follows from \eqref{equ2ndestimate} and the blow-up argument (see the details in \cite{fengzhengsasaki1}).

Given \eqref{equzeroestimate}, \eqref{equ2ndestimate} and \eqref{equ1stestimate}, $C^{2,\alpha}$-estimate for some $0<\alpha<1$ follows from the Evans-Krylov theory \cite{evanscpam1982,krylov1982,trudingertransactionams1983} (see also \cite{twwycvpde}).

Differentiating the equations and using the Schauder theory (see for example \cite{gt1998}), we then deduce uniform a priori $C^k$ estimates for all $k\geq 0.$

Finally, by routin arguments, the proof of Lemma \ref{taubinyau} is completed.
\end{proof}
\begin{lem}
\label{lemwuyauinventioneprop9}
Let $(X,\mathcal{F})$ be a closed oriented$,$ taut$,$ transverse K\"ahler foliated manifold$,$ where $\mathcal{F}$ is the foliation with complex codimension $n,$ and let $\omega$ and $\hat{\omega}$ be transverse K\"ahler metrics such that $\omega$ has transverse holomorphic sectional curvature bounded above by a constant $-\kappa\leq 0,$ and that $\hat{\omega}$ satisfies
\begin{equation}
\label{assum}
\mathrm{Ric}(\hat{\omega})\geq -\lambda\hat{\omega}+\nu\omega,
\end{equation}
for some constants $\lambda,\nu>0$. Then one can infer
\begin{equation}
\label{tdeltatrace}
\Delta_{\hat{\omega}}\log\mathrm{tr}_{\hat{\omega}}{\omega}\geq\left(\frac{n+1}{2n}\kappa+\frac{\nu}{n}\right)\ \mathrm{tr}_{\hat{\omega}}{\omega}-\lambda.
\end{equation}
In particular$,$ there holds
\begin{equation}
\label{upperbdttrace}
\sup_X \mathrm{tr}_{\hat{\omega}}{\omega}\leq \frac{\lambda}{\frac{n+1}{2n}\kappa+\frac{\nu}{n}}.
\end{equation}
\end{lem}
\begin{proof}
This a transverse version of \cite[Proposition 9]{wuyauinventiones}
(cf.\cite{wuyauzhengmrl2009,wongwuyauproceedingams2012}) which is an application of Yau's Schwarz Lemma \cite{yauajm1978} and Royden's trick \cite{roydencmh1980}. We refer to \cite{ps09,weinkovekrfpark} for the simplified calculation.

In the distinguished chart $(U;t^1,\cdots,t^k,z^1,\cdots,z^n)$, we work with the qualities of $(z^1,\cdots,z^n).$ Hence \eqref{tdeltatrace} follows from the same calculation  in the proof of \cite[Proposition 9]{wuyauinventiones} and \cite[Lemma 2.1]{tosattiyangjdg}. At the point where $\log\mathrm{tr}_{\hat{\omega}}{\omega}$ attain its maximum, there holds
$$
\Delta_{\hat{\omega}}\log\mathrm{tr}_{\hat{\omega}}{\omega}\leq 0,
$$
which yields \eqref{upperbdttrace}.
\end{proof}

\begin{lem}
\label{wuyauinventioneprop8}
Let $(X,\mathcal{F})$ be a closed oriented$,$ taut$,$ transverse K\"ahler foliated manifold$,$ where $\mathcal{F}$ is the foliation with complex codimension $n,$ and $\omega$ denote a transverse K\"ahler metric. If the normal canonical line bundle $K_{X/\mathcal{F}}$ is transverse nef$,$ then there hold
\begin{enumerate}
\item
\label{wuyauprop81} For each $\varepsilon>0,$ there exists a smooth basic function $u_\varepsilon\in C^\infty(X/\mathcal{F},\mathbb{R})$ such that
    \begin{equation*}
    \omega_\varepsilon^n=e^{u_\varepsilon}\omega^n,\quad \mbox{with}\quad
    \omega_\varepsilon:=-\mathrm{Ric}(\omega)
    +\sqrt{-1}\partial\bar\partial u_\varepsilon+\varepsilon\omega>0.
    \end{equation*}
Moreover$,$ there holds
\begin{equation}
\label{eq:lowRic}
\mathrm{Ric}(\omega_{\varepsilon}) = - \omega_{\varepsilon} + \varepsilon \omega \geq - \omega_{\varepsilon},
\end{equation}
and
\begin{equation}
\label{eq:uppu}
\sup_X u_{\varepsilon} \le C,
\end{equation}
where the constant $C>0$ depends only on $\omega$ and $n$.
\item \label{it:nefL2} For each $k = 1, \cdots, n,$ there holds
\begin{equation}
\label{equit:nefl2}
\int_X \left(-c_1(X/\mathcal{F})\right)^k \wedge  \omega^{n-k}\wedge\chi \ge e^{ \frac{(k-n)C}{n}} \int_X \left(-c_1(X/\mathcal{F})\right)^n \wedge\chi\ge 0,
\end{equation}
where the constant $C$ is the same as that in \eqref{eq:uppu}.
\end{enumerate}
\end{lem}
\begin{proof}
This is a transverse version of \cite[Proposition 8]{wuyauinventiones}. We adapt the idea in \cite{wuyauinventiones} to prove it.

We show Item \eqref{wuyauprop81}. For each $\varepsilon>0$, there exists a smooth basic function $f_\varepsilon\in C^\infty(X/\mathcal{F},\mathbb{R})$ such that
\begin{equation}
\label{omegafvarepsilon}
\omega_{f_\varepsilon}:=-\mathrm{Ric}(\omega)+\sqrt{-1}\partial\bar\partial f_\varepsilon+\varepsilon\omega>0
\end{equation}
since $K_{X/\mathcal{F}}$ is transverse nef.

Fix $\varepsilon>0$. It follows from Lemma \ref{taubinyau} that there exists a unique $v_\varepsilon\in C^\infty(X/\mathcal{F},\mathbb{R})$ such that
\begin{equation}
\label{taubinyauapp}
(\omega_{f_\varepsilon}+\sqrt{-1}\partial\bar\partial v_\varepsilon)^n=e^{v_\varepsilon+f_\varepsilon}\omega^n,\quad \omega_\varepsilon:=\omega_{f_\varepsilon}+\sqrt{-1}\partial\bar\partial v_\varepsilon>0.
\end{equation}
A direct calculation, together with \eqref{chernricciform}, \eqref{omegafvarepsilon} and \eqref{taubinyauapp}, yields that
\begin{align*}
\mathrm{Ric}(\omega_\varepsilon)
=&\mathrm{Ric}(\omega)
-\sqrt{-1}\partial\bar\partial f_\varepsilon
-\sqrt{-1}\partial\bar\partial v_\varepsilon\\
=&\varepsilon\omega-\omega_{f_\varepsilon}
-\sqrt{-1}\partial\bar\partial v_\varepsilon\\
=&-\omega_\varepsilon+\varepsilon\omega.
\end{align*}
We set $u_\varepsilon:=f_\varepsilon+v_\varepsilon$. It follows from the maximum principle, \eqref{omegafvarepsilon} and \eqref{taubinyauapp} that
\begin{equation*}
\sup_Xu_\varepsilon\leq C:=\log\left(\frac{(\varepsilon_0\omega-\mathrm{Ric}(\omega))^n}{\omega^n}\right),\quad \forall\;\varepsilon<\varepsilon_0,
\end{equation*}
which yields \eqref{eq:uppu}.

We show Item \eqref{it:nefL2}. Since $K_{X/\mathcal{F}}$ is transverse nef, it follows from Lemma \ref{basicstokes} and  \eqref{taubinyauapp} that
\begin{equation*}
\int_X\omega_{\varepsilon}^n\wedge\chi
=\int_X(-\mathrm{Ric}(\omega)+\varepsilon\omega)^n\wedge\chi>0,\quad \forall\;\varepsilon>0,
\end{equation*}
which yields that
\begin{equation}
\int_X\left(-c_1(X/\mathcal{F})\right)^n\wedge\chi=\lim_{\varepsilon \to 0^+}
\int_X\omega_{\varepsilon}^n\wedge\chi\geq 0.
\end{equation}
We set
$$
\sigma_k = \frac{ \omega_{\varepsilon}^k \wedge \omega^{n-k}}{\omega_\varepsilon^n}, \quad 0 \le k \le n.
$$
It follows from  Maclaurin's inequality that
$$
   \sqrt[k]{\sigma_k}\geq \sqrt[n]{\sigma_n},
$$
which, together with \eqref{eq:uppu} and \eqref{taubinyauapp}, yields that
\begin{equation}
\label{maclaurinapp}
\frac{\omega_\varepsilon^k\wedge\omega^{n-k}}{\omega_\varepsilon^n}
\geq \left(\frac{\omega_\varepsilon^n}{\omega^n}\right)^{\frac{k}{n}}
\left(\frac{\omega^n}{\omega_\varepsilon^{n}}\right)
=\left(e^{u_\varepsilon}\right)^{\frac{k}{n}-1}\geq e^{\frac{(k-n)C}{n}}.
\end{equation}
It follows from Lemma \ref{basicstokes}, \eqref{taubinyauapp} and \eqref{maclaurinapp} that
\begin{equation}
\label{maclaurinapp2}
\int_X(-\mathrm{Ric}(\omega)+\varepsilon\omega)^k\wedge\omega^{n-k}\wedge\chi
\geq  e^{\frac{(k-n)C}{n}}\int_X(-\mathrm{Ric}(\omega)+\varepsilon\omega)^n\wedge\chi,
\quad \forall\;\varepsilon>0.
\end{equation}
Letting $\varepsilon\to 0^+$ in \eqref{maclaurinapp2} shows \eqref{equit:nefl2}.
\end{proof}

\begin{lem}
\label{mainthmtosattiyangjdg}
Let $(X,\mathcal{F})$ be a closed oriented$,$ taut$,$ transverse K\"ahler foliated manifold$,$ where $\mathcal{F}$ is the foliation with complex codimension $n$ and $\omega$ is a transverse K\"ahler metric with nonpositive transverse holomorphic sectional curvature. Then the normal canonical bundle $K_{X/\mathcal{F}}$ is transverse nef.
\end{lem}
\begin{proof}
This is a transverse version of \cite[Theorem 1.1]{tosattiyangjdg} and we adapt the idea from \cite{tosattiyangjdg} to show it. If $K_{X/\mathcal{F}}$ is not transverse nef, then there exists $\varepsilon_0>0$ such that the class
$\varepsilon_0[\omega]-2\pi c_1(X/\mathcal{F})$ is transverse nef but not K\"ahler. As in the proof of Lemma \ref{thmwuyauinventionethm7}, it follows from Item \eqref{wuyauprop81} of Lemma \ref{wuyauinventioneprop8} that
for each $\varepsilon>0$ there exists transverse K\"ahler metric $\omega_{\varepsilon} := \varepsilon \omega -\mathrm{Ric}(\omega) + \sqrt{-1}\partial\bar\partial u_\varepsilon$ satisfying $\omega^n_{\varepsilon} =  e^{u_{\varepsilon}} \omega^n$
and
\begin{equation}
\label{ricomegavarepsilon2}
\textup{Ric}(\omega_{\varepsilon}) = - \omega_{\varepsilon} + (\varepsilon+\varepsilon_0) \omega, \quad \max_X u_{\varepsilon} \le C,\quad C^{-1}\omega\leq \omega_\varepsilon\leq C\omega
\end{equation}
where $C=C(\omega,n,\varepsilon_0)>0$ is a constant independent of $\varepsilon$.
Here we should use \eqref{upperbdttrace} with $\kappa=0,\lambda=1,\nu=(\varepsilon+\varepsilon_0)$.

As in the proof of Lemma \ref{thmwuyauinventionethm7},  there still holds
\begin{equation}
\label{tgaojieguji2}
\|\omega_\varepsilon\|_{C^k(X,\omega)}\leq C_k,
\end{equation}
where $C_k>0$ is a constant independent of $\varepsilon$.

It follows from \eqref{ricomegavarepsilon2},   \eqref{tgaojieguji2},   the Ascoli-Arzel\`a theorem and a diagonal argument that  there exists a sequence $\{ \varepsilon_i \}$ with $\lim_{i\to\infty}\varepsilon_i=0$ such that $\omega_{\varepsilon_i}$ converge smoothly to a transverse K\"ahler metric $\omega_0$ which satisfies
$$
[\omega_0]=\varepsilon_0[\omega]-2\pi c_1(X/\mathcal{F}),
$$
which contradicts to the fact that $\varepsilon_0[\omega]-2\pi c_1(X/\mathcal{F})$ is not transverse K\"ahler, as desired.
\end{proof}

\begin{lem}
\label{thmwuyauinventionethm7}
Let $(X,\mathcal{F})$ be a closed oriented$,$ taut$,$ transverse K\"ahler foliated manifold$,$ where $\mathcal{F}$ is the foliation with complex codimension $n,$ and $\omega$ denote the transverse K\"ahler metric  with negative transverse holomorphic sectional curvature with upper bound $-\kappa<0$. Then there exists a smooth basic function $u \in C^\infty(X/\mathcal{F},\mathbb{R})$ such that $ \omega_u :=-\mathrm{Ric}(\omega)+\sqrt{-1}\partial\bar\partial u$ is the transverse K\"ahler-Einstein metric  with $\mathrm{Ric}(\omega_u)=-\omega_u$.
\end{lem}
\begin{proof}
This is a transverse version of \cite[Theorem 9]{wuyauinventiones} and we adapt the idea from \cite{wuyauinventiones} to show it.

It follows from Lemma \ref{wuyauinventioneprop8}\eqref{wuyauprop81} and Lemma \ref{mainthmtosattiyangjdg} that
for each $\varepsilon>0$ there exists transverse K\"ahler metric $\omega_{\varepsilon} := \varepsilon \omega -\mathrm{Ric}(\omega) + \sqrt{-1}\partial\bar\partial u_\varepsilon$ satisfying $\omega^n_{\varepsilon} =  e^{u_{\varepsilon}} \omega^n$
and
\begin{equation}
\label{ricomegavarepsilon}
\textup{Ric}(\omega_{\varepsilon}) = - \omega_{\varepsilon} + \varepsilon \omega, \quad \max_X u_{\varepsilon} \le C,
\end{equation}
where $C=C(\omega,n)>0$ is a constant independent of $\varepsilon$.

It follows from \eqref{upperbdttrace} with $\lambda=1,\nu=\varepsilon$ and \eqref{ricomegavarepsilon} that
\begin{equation}
\label{tromegavarepsilonomega}
\mathrm{tr}_{\omega_\varepsilon}\omega\leq \frac{2n}{(n+1)\kappa}.
\end{equation}
The uniform upper bound for $u_\varepsilon<C$ yields that
\begin{equation}\label{tromegavarepsilonomega1}
\sup_X\frac{\omega_\varepsilon^n}{\omega^n}\leq C.
\end{equation}
 Thanks to \eqref{tromegavarepsilonomega}, \eqref{tromegavarepsilonomega1} and the elementary inequality
$$\mathrm{tr}_{\omega}{\omega_\varepsilon}\leq \frac{1}{(n-1)!}(\mathrm{tr}_{\omega_\varepsilon}{\omega})^{n-1}\frac{\omega_\varepsilon^n}{\omega^n},$$
we can deduce that
\begin{equation}
\label{tromegaomegavarepsilon}
\sup_X\mathrm{tr}_{\omega}{\omega_\varepsilon}\leq C,
\end{equation}
where $C=C(\omega,n,\kappa)>0$ is a constant independent of $\varepsilon$.

It follows from \eqref{tromegavarepsilonomega} and \eqref{tromegaomegavarepsilon} that
\begin{equation}
C^{-1}\omega\leq \omega_\varepsilon\leq C\omega,
\end{equation}
which yields that
\begin{equation}
\label{omegaomegavarepsilondengjia}
\inf_Xu_\varepsilon\geq -C,
\end{equation}
where $C=C(\omega,n,\kappa)>0$ is a constant independent of $\varepsilon$.

We claim
\begin{equation}
\label{tgaojieguji}
\|\omega_\varepsilon\|_{C^k(X,\omega)}\leq C_k,
\end{equation}
where $C_k>0$ is a constant independent of $\varepsilon$.
Indeed,  in the distinguished chart $$(U;t^1,\cdots,t^k,z^1,\cdots,z^n),$$ we work with the qualities of $(z^1,\cdots,z^n).$ Hence \eqref{tdeltatrace} follows from the same argument in \cite{tosattiyangjdg} following the work of Yau \cite{yau1978}.

It follows from \eqref{omegaomegavarepsilondengjia},   \eqref{tgaojieguji},   the Ascoli-Arzel\`a theorem and a diagonal argument that  there exists a sequence $\{u_{\varepsilon_i}\}$ with $\lim_{i\to\infty}\varepsilon_i=0$ converge smoothly to a smooth basic function $u\in C^\infty(X/\mathcal{F},\mathbb{R})$ which satisfies
\begin{equation*}
\left(-\mathrm{Ric}(\omega) + \sqrt{-1}\partial\bar\partial u\right)^n =  e^{u } \omega^n,\quad -\mathrm{Ric}(\omega) + \sqrt{-1}\partial\bar\partial u>0.
\end{equation*}
This shows that $\omega_u:=-\mathrm{Ric}(\omega) + \sqrt{-1}\partial\bar\partial u>0$ satisfies $\mathrm{Ric}(\omega_u)=-\omega_u$, as required.
\end{proof}

Now we finish the proofs of Theorems \ref{tmainthmtyjdg} and \ref{mainthm2}.

\begin{proof}[Proof of Theorem \ref{tmainthmtyjdg}]
The conclusion follows from Lemmas \ref{mainthmtosattiyangjdg} and \ref{thmwuyauinventionethm7}.
\end{proof}

\begin{proof}[Proof of Theorem \ref{mainthm2}]
It follows from Theorem \ref{tmainthmtyjdg} that there exists a transverse K\"ahler metric $\tilde{\omega}$ such that $2\pi c_1(X/\mathcal{F})=[\mathrm{Ric}(\tilde{\omega})]=[-\tilde{\omega}].$ Since $c_1(T_{\mathcal{F}})=0$, one infers from the adjunction formula that
$$
2\pi c_1(X)=2\pi c_1(X/\mathcal{F})=[-\tilde{\omega}].
$$
Hence $c_1(X)$ is represented by a semi-negative closed real $(1,1)$ form and $c_1(X)^n\not=0.$ Then Theorem \ref{mainthm2} follows from \cite[Theorem 1.2]{touzettoulouse2010}.
\end{proof}
\begin{appendix}
\section{Preliminaries for Distribution and Current}
\label{section:appendix}
In this appendix, we collect preliminaries for distribution and current in order to deduce the transverse versions of the Poincar\'e Lemma and  the Dolbeault-Grothendieck lemma.

Let $X$ be a smooth oriented  differential manifold with $\dim_{\mathbb{R}}X=n$ and countable base.  We first introduce a topology on the space of differential forms $C^s\left(X,\bigwedge^pT_X^*\right)$. Let $U\subset X$ be a coordinate open set and $u$ a $p$ form on $X$, written $u(x)= \sum_{i_1<\cdots<i_p} u_I(x)\mathrm{d} x^I$ on $U$ with multi-indices of the type $I=(i_1,\cdots,i_p)$. To every compact subset $K\subset U$ and every integer $s\in\mathbb{N}$, we associated a semi-norm
\begin{equation}
\label{equ1-2.1}
p_K^s(u)=\sup_{x\in K}\max_{|I|=p,\,|\alpha|\leq s}\left|\partial^\alpha u_I(x)\right|,
\end{equation}
where $\alpha=(\alpha_1,\cdots,\alpha_n)$ runs over $\mathbb{N}^n$ and $\partial^\alpha=\partial^{|\alpha|}/\partial (x^1)^{\alpha_1}\cdots\partial (x^n)^{\alpha_n}$ is a derivation of order $|\alpha|=\alpha_1+\cdots+\alpha_n$.
\begin{defn}
\label{defn1-2.2}
We introduce as follows spaces of $p$ forms on manifolds.
\begin{enumerate}
\item\label{defn1-2.2a}We denote by $\mathscr{E}^p(X)$ $($resp. ${}^s\mathscr{E}^p(X)$$)$ the space $C^\infty(X,\bigwedge^p T_X^*)$ $($resp. the space $C^s(X,\bigwedge^pT_X^*)$$)$, equipped with the topology defined by all semi-norms $p_L^s$ when $s,\,L,\,U$ vary $($resp. when $L,\,U$ vary$)$. We use the notation $\mathscr{E}(X):=\mathscr{E}^0(X)$ and ${}^s\mathscr{E}(X):={}^s\mathscr{E}^0(X)$.
\item\label{defn1-2.2b}If $K\subset X$ is a compact subset$,$ $\mathscr{D}^p(K)$ will denote the subspace of elements $u\in\mathscr{E}^p(X)$ with support contained in $K$$,$ together with the induced topology$;$ $\mathscr{D}^p(X)$ will stand for the set of all elements with compact support$,$ i.e.$,$ $\mathscr{D}^p(X):=\bigcup_{K}\mathscr{D}^p(K)$. We use the notation $\mathscr{D}(X):=\mathscr{D}^0(X),\;\mathscr{D}(K):=\mathscr{D}^0(K)$.
\item\label{defn1-2.2c}The spaces of $C^s$ forms ${}^s\mathscr{D}^p(X)$ and ${}^s\mathscr{D}^p(X)$ are defined similarly. We use the notation ${}^s\mathscr{D}(X):={}^s\mathscr{D}^0(X),\;{}^s\mathscr{D}(K):={}^s\mathscr{D}^0(K)$.
\end{enumerate}
\end{defn}
Since $X$ is separable, the topology of $\mathscr{E}^p(X)$ can be defined
by means of a countable set of semi-norms $p_K^s$, hence
$\mathscr{E}^p(X)$ (and likewise
${}^s\mathscr{E}^p(X)$ is a
Fr\'echet space. The topology of
${}^s\mathscr{D}^p(K)$ is induced by any finite set of semi-norms $p_{K_j}^s$
such that the compact sets $K_j$ cover $K$; hence
${}^s\mathscr{D}^p(K)$ is a Banach space. It should be
observed however that $\mathscr{D}^p(X)$ is not a Fr\'echet space; in fact $\mathscr{D}^p(X)$ is dense in $\mathscr{E}^p(X)$
and thus non complete for the induced topology.
\begin{defn}
Let $X$ be a complex manifold with $\dim_{\mathbb{C}}X=n$ and countable base. Then  there are decompositions
\begin{equation*}
\mathscr{D}^k(X,\mathbb{C})=\bigoplus_{p+q=k}\mathscr{D}^{p,q}(X,\mathbb{C}),\quad
\mathscr{D}'_k(X,\mathbb{C})=\bigoplus_{p+q=k}\mathscr{D}'_{p,q}(X,\mathbb{C}).
\end{equation*}
The space $\mathscr{D}'_{p,q}(X,\mathbb{C})$ is called the space of currents of bidimension $(p,q)$ and bidegree
$(n-p,n-q)$ on $X$, and is also denoted $\mathscr{D}'^{n-p,n-q}(X,\mathbb{C})$.
\end{defn}
\subsection{Distribution}
A distribution on $U\subset\mathbb{R}^n$ is a continuous linear form on $\mathscr{D}(U)$.
\begin{defn}
\label{romanddefn2.3.2}
A distribution on $U\subset\mathbb{R}^n$ is a linear form $u$ on $\mathscr{D}(U)$ such that for each compact set $K\subset U$ there exist constants $C$ and $s\in \mathbb{N}$ such that
\begin{equation*}
|u(\varphi)|\leq Cp_K^s(\varphi),\quad\forall\;\varphi\in\mathscr{D}(K).
\end{equation*}
The set of all distributions in $U$ is denoted by $\mathscr{D}'(U)$.
\end{defn}
We also use the notation $\langle u,\varphi\rangle:= u(\varphi),\; \forall\;\varphi\in\mathscr{D}(U)$.
The continuity condition in Definition \ref{romanddefn2.3.2} is often stated as a sequential continuity.
\begin{thm}
Let $U\subset\mathbb{R}^n$ be an open set, and $u$   a complex valued linear form on $\mathscr{D}(U)$. Then $u$ is a distribution if and only
$u(\varphi_j)\to 0$ when $j\to\infty$ for every sequence $(\varphi_j)_{j\in\mathbb{N}^*}$ converging to $0$ in the sense that $p_K^s(\varphi_j)\to 0$ as $j\to \infty$ for each fixed $s$ and $\mathrm{supp}\varphi_j\subset K$ for all $j$ with $K\subset U$ compact$;$
if and only if there exist continuous functions $\rho_\alpha$ for each $\alpha\in\mathbb{N}^n$ such that
$$
|u(\varphi)|\leq \sum_{\alpha\in\mathbb{N}^n} |\rho_\alpha \partial^\alpha\varphi|,\quad \forall\;\varphi\in\mathscr{D}(U),
$$
and on each compact set $K\subset U$ all but finite number of the functions $\rho_\alpha$ vanish identically.
\end{thm}
\begin{proof}
See \cite[Theorem 2.1.4 \& 2.1.5]{hormander}.
\end{proof}
A distribution is determined by the restriction to the sets in an open covering.
\begin{thm}
\label{hormanderthm2.2.1-2.2.4}
Let $U\subset\mathbb{R}^n$ and $u\in\mathscr{D}'(U)$. Then we have
\begin{enumerate}
\item If for each point of $U$ there exists a neighborhood to which the restriction of $u$ is  $0,$ then $u=0$.
\item Let $\{U_i\}_{i\in I}$ be an arbitrary open covering of $U$. Then if $u_i\in\mathscr{D}'(U_i)$ and $u_i=u_j$ on $U_i\cap U_j$ for all $i,j\in I,$ then
    there exists one and only one $u\in\mathscr{D}'(U)$ such that $u_i$ is the restriction of $u$ to $U_i$ for each $i\in I.$
\end{enumerate}
\end{thm}
\begin{proof}
See \cite[Theorem 2.2.1 \& Theorem 2.2.4]{hormander}.
\end{proof}
\begin{thm}
Let $U\subset\mathbb{R}^n$ and $u\in\mathscr{D}'(U)$. Then if $u(\varphi)\geq 0$ for each non-negative $\varphi\in\mathscr{D}(U)$, then $u$ is a positive regular Borel measure.
\end{thm}
\begin{proof}
See  \cite[Theorem 2.1.7]{hormander} and \cite[Theorem 2.14 and Theorem 2.18]{rudinrealanalysis}.
\end{proof}
One can of course multiply a distribution $u$ in $\mathscr{D}'(U)$ by a smooth function $f\in C^\infty(U)$, and define partial
derivatives $\partial^\alpha u$ of a distribution $u$ by the formulae
\begin{align}
\label{defnchengdis}
&\langle fu,\varphi\rangle:=\langle u,f\varphi\rangle,\\
\label{defndaodis}
&\langle \partial^\alpha u,\varphi\rangle:=(-1)^{|\alpha|}\langle u,\partial^\alpha\varphi\rangle,\quad\forall\;\alpha\in\mathbb{N}^n,
\;\forall\;\varphi\in\mathscr{D}(U).
\end{align}
Indeed, these linear forms defined in \eqref{defnchengdis} and \eqref{defndaodis} are continuous on $\mathscr{D}(U)$ and hence are well defined.
\begin{prop}\label{ramondprop4.1.1}
Let $U\subset \mathbb{R}^n$ be an open set, and $u\in \mathscr{D}'(U)$. Let also $\varphi\in C^\infty(U\times\mathbb{R}^q)$. If
there exists a compact set $K\subset U$ such that $\mathrm{supp}\varphi\subset K\times \mathbb{R}^q$, then the function
$$
G:\;\mathbb{R}^q\to\mathbb{R},\quad y\mapsto \langle u,\varphi(\cdot,y)\rangle
$$
is $C^\infty$ , and there holds
$$
\partial^\alpha G(y)=\langle u,\partial_y^\alpha\varphi(\cdot,y)\rangle,\quad \forall\;\alpha \in\mathbb{N}^q.
$$
\end{prop}
\begin{proof}
See \cite[Propositon 4.1.1]{ramond2015} or \cite[Theorem 2.1.3]{hormander}.
\end{proof}
\begin{rem}
\begin{enumerate}
\item We have written $\langle u, \varphi(\cdot,y)\rangle$ in place of  $\langle u, \varphi_y\rangle$, where $\varphi_y\in C_0^\infty(\mathbb{R}^n)$
 is the
function given by $\varphi_y(x)=\varphi(x,y)$.
\item The assumption $\mathrm{supp}\varphi\subset K\times \mathbb{R}^q$  means that for any $y\in \mathbb{R}^q$, the support of $\varphi_y$ is included
in $K$. It holds in particular when $\varphi \in C^\infty_0(\Omega\times\mathbb{R}^q)$.
\item For a regular distribution $u=u_f$ with $f\in L^1_{\mathrm{loc}}(\Omega)$, we have $$
    G(y)=\int f(x)\varphi(x,y)\mathrm{d}x,
    $$
so that, under the above assumptions, we get $G\in C^\infty(\mathbb{R}^q)$ and
$$
\partial^\alpha G(y)=\int f(x)\partial_y^\alpha\varphi(x,y)\mathrm{d}x,\quad \forall\;\alpha \in\mathbb{N}^q.
$$
\end{enumerate}
\end{rem}
\begin{prop}
\label{ramondprop4.1.3}
Let $U\subset\mathbb{R}^k$ be an open set, and $u\in\mathscr{D}'(U)$. Let also $\varphi\in C^\infty(U\times \mathbb{R}^q)$.
Then
\begin{equation*}
\int_{\mathbb{R}^q}\langle u,\varphi(\cdot,y)\rangle\mathrm{d}y=\langle u,\int_{\mathbb{R}^q}\varphi(\cdot,y)\mathrm{d}y\rangle.
\end{equation*}
\end{prop}
\begin{proof}
See \cite[Proposition 4.1.3]{ramond2015}. 
\end{proof}
Let $f:\;\mathbb{R}^p\to\mathbb{C}$ and $g:\;\mathbb{R}^q\to\mathbb{C}$ be two functions. Then the function $f\otimes g$ is
defined on $\mathbb{R}^{p+q}$ by
$$
(f\otimes g)(x_1,x_2):=f(x_1)g(x_2),\quad \forall\;x=(x_1,x_2)\in\mathbb{R}^{p+q} \mbox{ with }x_1\in\mathbb{R}^p,\;x_2\in\mathbb{R}^q.
$$
This function is called the tensor product of $f$ and $g$.
\begin{prop}
Let $U\subset\mathbb{R}^n$ be an open set  and $u\in\mathscr{D}'(U)$. Then
\begin{equation*}
\partial_ju=0,\;\forall\;1\leq j\leq n \Leftrightarrow \exists\;C\in\mathbb{C}\; \mbox{such that}\;\langle u,\varphi\rangle=C\int\varphi,\quad \forall\;\varphi\in\mathscr{D}(U).
\end{equation*}
\end{prop}
\begin{proof}
See \cite[Proposition 4.2.3]{ramond2015}.
\end{proof}
The tensor product of distributions is defined by
\begin{prop}
\label{ramondprop4.2.4}
Let $(u_1,u_2)\in\mathscr{D}'(U_1)\times \mathscr{D}'(U_2)$. Then one has
\begin{enumerate}
\item for each $\varphi\in\mathscr{D}(U_1\times U_2),$ the function $\psi:\;x\mapsto\langle u_2,\varphi(x,\cdot)$ belongs to $\mathscr{D}(U_1);$
\item the linear form $u:\;\mathscr{D}(U_1\times U_2)\ni\varphi\mapsto \langle u_1,\psi\rangle$ is continuous in $\mathscr{D}(U_1\times U_2)$, and it is called the tensor product, denoted by $u=:u_1\otimes u_2$, of $u_1$ and $u_2$. It is the only distribution in $\mathscr{D}'(U_1\times U_2)$ such that
    $$
    \langle u,\varphi_1\otimes \varphi_2\rangle=\langle u_1,\varphi_1\rangle\langle u_2,\varphi_2\rangle,\quad\forall\;(\varphi_1,\varphi_2)\in \mathscr{D}(U_1)\times \mathscr{D}(U_2).
    $$
\end{enumerate}
\end{prop}
\begin{proof}
See \cite[Proposition 4.2.4 \& Proposition 4.2.5]{ramond2015}.
\end{proof}
\begin{prop}
\label{ramondexercise4.2.7}
Let $(u_1,u_2)\in\mathscr{D}'(U_1)\times \mathscr{D}'(U_2),$ where $U_1\subset\mathbb{R}^k$ and $U_2\subset\mathbb{R}^n$ are open sets. Then for $j\in\{1,\cdots,n+k\},$ one has
\begin{equation*}
\partial_j(u_1\otimes u_2)
=
\left\{\begin{array}{ll}
(\partial_j u_1) \otimes u_2,&\quad\mbox{for}\quad 1\leq j\leq k,   \\
u_1\otimes(\partial_j u_2), & \quad\mbox{for}\quad k+1\leq j\leq n+k.
\end{array}
\right.
\end{equation*}
\end{prop}
\begin{proof}
See \cite[Proposition 4.2.7]{ramond2015}.
\end{proof}
\begin{prop}
\label{ramondexercise4.2.9}
Let $u\in \mathscr{D}'(U_1\times U_2),$ where $U_1\subset\mathbb{R}^k$ and $U_2\subset\mathbb{R}^n$ are open sets. Then the following two assertions are equivalent$:$
\begin{enumerate}
\item \label{ramondexercise4.2.9item1}$\partial_j u=0$ for all $j\in\{1,\cdots,k\};$
\item\label{ramondexercise4.2.9item2}There exists $v\in \mathscr{D}'(U_2)$ such that $u=1\otimes v$.
\end{enumerate}
\end{prop}
\begin{proof}
$\eqref{ramondexercise4.2.9item2} \Rightarrow \eqref{ramondexercise4.2.9item1} $ follows from Proposition \ref{ramondexercise4.2.7}.

$\eqref{ramondexercise4.2.9item1} \Rightarrow \eqref{ramondexercise4.2.9item2} $ follows from \cite[Theorem 3.1.4$'$]{hormander} for $n=1$, and for general $n$, we first claim that for each $\chi\in \mathscr{D}(U_1\times U_2)$ there holds
\begin{equation*}
\chi\in \mathrm{Im}\left(\partial_{1}\circ\cdots\circ\partial_{k}\right) =\bigcap_{\ell=1}^k \mathrm{Im}\partial_{\ell}
\end{equation*}
if and only if
\begin{equation}
\label{equhorvath1966exercise4-3-b}
\int_{\mathbb{R}^k}\chi(t^{1},\cdots,t^{k},x^{k+1},\cdots,x^{k+n})
\mathrm{d}t^{1}\wedge\cdots\wedge \mathrm{d}t^{k}=0.
\end{equation}
Indeed, for `only if' direction, we assume that $\chi\in \mathrm{Im}\left(\partial_{1}\circ\cdots\circ\partial_{k}\right)\cap\mathscr{D}(U_1\times U_2)$, i.e., $\chi$ satisfies
\begin{equation}
\label{equhorvath1966exercise4-3-b1}
\frac{\partial^k\psi}{\partial x^{1}\cdots\partial x^{k}}=\chi
\end{equation}
for $\psi\in \mathscr{D}(U_1\times U_2)$. Equality \eqref{equhorvath1966exercise4-3-b} holds by a direct calculation.

For the `if' direction, the function
\begin{equation*}
\psi(x):=\int_{-\infty}^{x^{1}}\cdots\int_{-\infty}^{x^{k}}\chi(t^{1},\cdots,t^{k},
x^{k+1},\cdots,x^{k+n})
\mathrm{d}t^{1}\wedge\cdots\wedge \mathrm{d}t^{k}
\end{equation*}
satisfies \eqref{equhorvath1966exercise4-3-b1}. In addition, $\psi\in\mathscr{D}(U_1\times U_2)$ if \eqref{equhorvath1966exercise4-3-b} holds.

We choose $\psi_0\in \mathscr{D}(U_1)$ with $\int_{U_1}\psi_0=1$ and define
\begin{equation*}
\langle v,\chi\rangle:=\langle u,\chi_0\rangle,\quad \forall\; \chi\in\mathscr{D}(U_2)
\end{equation*}
with
$$
\chi_0:=\chi(x^{k+1},\cdots,x^{k+n})\psi_0(x^{1},\cdots,x^{k}).
$$
A direct check yields that $v\in\mathscr{D}'(U_2)$. If $\phi\in \mathscr{D}(U_1\times U_2)$ and we set
$$
I(\phi):=\int_{\mathbb{R}^k}\phi(t^{1},\cdots,t^{k},x^{k+1},\cdots,x^{k+n})
\mathrm{d}t^{1}\wedge\cdots\wedge \mathrm{d}t^{k},
$$
and have that
$$
\int_{\mathbb{R}^n}(\phi(t^1,\cdots,t^k,x^{k+1},\cdots,x^{k+n})-I(\phi)\psi_0)
\mathrm{d}t^{1}\wedge\cdots\wedge \mathrm{d}t^{k}=0.
$$
Hence we have $\phi-I(\phi)\psi_0\in  \mathrm{Im}\left(\partial_{1}\circ\cdots\circ\partial_{k}\right),$
which, together with the fact that $\partial_ju=0$ for $j=1,\cdots,k$, yields that
$$
\langle u,(\phi-I(\phi)\psi_0)\rangle=0,
$$
i.e.,
\begin{align*}
\langle u, \phi\rangle=&\langle u,(I(\phi))_0\rangle=\langle v,I(\phi)\rangle\\
=&\langle v,\int_{\mathbb{R}^k}\phi(x^1,\cdots,x^{k+n})
\mathrm{d}x^{1}\wedge\cdots\wedge\mathrm{d}x^{k}\rangle\\
=&\int_{\mathbb{R}^k}\langle v,\phi\rangle\mathrm{d}x^{1}\wedge\cdots\wedge\mathrm{d}x^{k}
\end{align*}
by Proposition \ref{ramondprop4.1.3}, as desired.
\end{proof}

Let us recall the pullback of a distribution by some smooth map. 
Let $U_j\subset\mathbb{R}^{n_j}$ for $j=1,2$ be open sets, and $f:\;U_1\to U_2$  a smooth map such that $f'(x)$ is surjective for each point $x\in\mathbb{R}^{n_1}$ (which yields that $n_1\geq n_2$). Then for each $x_0\in U_1$ fixed we choose a smooth map $g:\;U_1\to\mathbb{R}^{n_1-n_2}$ (e.g., a linear map) such that the direct sum $f\oplus g$
$$
U_1\ni x\mapsto (f(x),g(x))\in \mathbb{R}^{n_1}=\mathbb{R}^{n_2}\times \mathbb{R}^{n_1-n_2}
$$
has a bijective differential at $x_0$. By the inverse theorem there exists an open neighborhood $V_1\subset U_1$ such that the restriction of $f\oplus g$ to $V_1$ is a diffeomorphism on an open neighborhood $V_2$ of $(f(x_0),g(x_0))$. We denote by $h$ its inverse. For any $\phi\in\mathscr{D}(V_1)$, we define
\begin{equation}
\label{hormanderequ6.1.1}
(f^*u)(\phi):=(u\otimes \mathbf{1}_{\mathbb{R}^{n_1-n_2}})(\Phi),\quad \Phi(y):=\phi(h(y))|\det h'(y)|,
\end{equation}
where $\mathbf{1}_{\mathbb{R}^{n_1-n_2}}$ is the function $1$ on $\mathbb{R}^{n_1-n_2}$. If in addition $u\in C^0(U_2)$, then we have $f^*u=u\circ f$.
\begin{thm}
\label{hormanderthm6.1.2}
Let $U_j\subset\mathbb{R}^{n_j}$ for $j=1,2$ be open sets$,$ and $f:\;U_1\to U_2$  a smooth map such that $f'(x)$ is surjective for each point $x\in\mathbb{R}^{n_1}$.
Then there exists a unique continuous linear map $f^*:\,\mathscr{D}'(U_2)\to \mathscr{D}'(U_1)$ such that $f^*u=u\circ f$ when $u\in C^0(U_2)$. It maps $\mathscr{D}'^k(U_2)$ into $\mathscr{D}'^k(U_1)$ for each $k$. We call $f^*u$ the pullback of $u$ by $f$.
\end{thm}
\begin{proof}
See \cite[Theorem 6.1.2]{hormander}.
\end{proof}
\begin{defn}
\label{hormanderdefn6.3.3}
Let $X$ be a  smooth oriented manifold with $\dim_{\mathbb{R}}X=n$ and countable basis and $\mathscr{A}$ an atlas on $X$ consisting of homeomorphism $\kappa$ of open set $U_\kappa\subset X$ to $\tilde U_\kappa\subset\mathbb{R}^n$. If for each atlas $\kappa$ one is given a distribution $u_\kappa\in\mathscr{D}'(\tilde U_\kappa)$ such that
\begin{equation}
\label{hormanderequ6.3.3}
u_{\kappa'}:=(\kappa\circ\kappa'^{-1})^*u_{\kappa}\quad \mbox{in}\quad \kappa'(U_\kappa\cap U_{\kappa'}),
\end{equation}
then the system $u_\kappa$ is called a distribution in $X$. The set of all distributions in $X$ is denoted by $\mathscr{D}'(X)$.
\end{defn}
The following theorem shows in particular that Definition \ref{hormanderdefn6.3.3} coincides with our previous one if $M$ is an open subset of $\mathbb{R}^n$.
\begin{thm}
\label{hormanderthm6.3.4}
Let $X$ be a smooth oriented manifold with $\dim_{\mathbb{R}}X=n$ and countable basis and $\mathscr{A}$ an atlas consisting of homeomorphism $\kappa$ of open set $U_\kappa\subset\mathbb{R}^n$ to $\tilde U_\kappa\subset\mathbb{R}^n$. If for each $\kappa\in\mathscr{A}$ we have a distribution $u_\kappa\in \mathscr{D}'(\tilde U_\kappa)$ 
and \eqref{hormanderequ6.3.3}
is valid when $\kappa$ and $\kappa'$ belong to $\mathscr{A}$, then there exists one and only one distribution $u\in \mathscr{D}'(X)$
such that $(\kappa^{-1})^*u =u_\kappa$
for each $\kappa\in\mathscr{A}$.
\end{thm}
\begin{proof}
See \cite[Theorem 6.3.4]{hormander}.
\end{proof}
\subsection{Current}
According to  De Rham \cite{derham55} currents are analogy with the
usual definition of distributions.
\begin{defn}
\label{defn1-2.3}
The space of currents of dimension $p$ $($or degree $m-p$$)$ on $X$ is the space $\mathscr{D}'_p(X)$ of linear forms $T$ on $\mathscr{D}^p(X)$ such that the restriction of $T$ to all subspaces $\mathscr{D}^p(K),\,K\subset\subset X$$,$ is continuous. The degree is indicated by raising the index$,$ hence we set
$$
\mathscr{D}'^{m-p}(X)=\mathscr{D}'_p(X):=\text{topological dual}\,\left(\mathscr{D}^p(X)\right)'
$$
The space ${}^s\mathscr{D}'_p(X)={}^s\mathscr{D}'^{m-p}(X):=\left({}^s\mathscr{D}^p(X)\right)'$ is defined similarly and is called the space of currents of order $s$ on $X$.
\end{defn}
Let $U$ be an open set in $\mathbb{R}^n$. Then a distribution $u$ can be seen as a current with degree $0$ by
\begin{equation*}
  \langle u,\varphi\mathrm{d}x^1\wedge\cdots\wedge\mathrm{d}x^n\rangle:=\langle u,\varphi\rangle,\quad \forall\;\varphi\in\mathscr{D}(U).
\end{equation*}
For each $T\in \mathscr{D}'_p(U)$ can be written as
\begin{equation*}
T:=\frac{1}{p!} a_{I}\mathrm{d}x^I
\end{equation*}
with $a_I\in\mathscr{D}'(U)$ defined by
\begin{equation*}
\langle a_I,\varphi\rangle=\langle u,\varphi\mathrm{d}x^1\wedge\cdots\wedge\mathrm{d}x^n\rangle
=\frac{1}{(n-p)!}\delta_{I_p,I_p^c}\langle T,\varphi\mathrm{d}x^{I_p^c}\rangle,
\end{equation*}
where $\delta_{I_p,I_p^c}$ is the  multi-index Kronecker delta. A current is determined by the restriction to the sets in an open covering and Theorem
\ref{hormanderthm2.2.1-2.2.4} holds for $p$ currents in $\mathscr{D}_p'(U)$. In addition, we have an analogy of  Theorem \ref{hormanderthm6.3.4} as follows.
\begin{thm}
\label{hormanderdefn6.3.3revisedforcurrent}
Let $X$ be a  smooth oriented manifold with $\dim_{\mathbb{R}}X=n$ and countable basis and $\mathscr{A}$ an atlas on $X$ consisting of homeomorphism $\kappa$ of open set $U_\kappa\subset X$ to $\tilde U_\kappa\subset\mathbb{R}^n$.
Then there exists a family of $p$ currents $(a_\kappa)_{\kappa\in\mathscr{A}}$ $($$1\leq p\leq n$$)$ with  $a_\kappa\in\mathscr{D}'_{n-p}(\tilde U_\kappa)$ given by
\begin{equation}
a_\kappa:= \frac{1}{p!}a_{\kappa;i_1,\cdots,i_p}\mathrm{d}x^{i_1}_\kappa\wedge \cdots\wedge \mathrm{d}x_\kappa^{i_p}
\end{equation}
such that
\begin{equation}
\label{hormanderequ6.3.3current}
a_{\kappa';j_1,\cdots,j_p}=\left((\kappa\circ \kappa'^{-1})^*a_{\kappa;i_1,\cdots,i_p}\right)\frac{\partial x_\kappa^{i_{1}}}{\partial x_{\kappa'}^{j_{1}}}\cdots
\frac{\partial x_\kappa^{i_p}}{\partial x_{\kappa'}^{j_p}}\quad \mbox{in}\quad \kappa'(U_\kappa\cap U_{\kappa'}),
\end{equation}
if and only if there exists unique $p$ current $T\in\mathscr{D}'_p(X)$ such that
\begin{equation}
\label{localcurrentdefn}
\langle a_\kappa,\varphi\rangle=\langle T,\kappa^*\varphi\rangle,\quad \forall\;\varphi\in\mathscr{D}^p(\tilde U_\kappa)
\end{equation}
for each $\kappa\in\mathscr{A}$.
\end{thm}
\begin{proof}
For `if' direction, let $T\in\mathscr{D}'_p(X)$ and $\psi:=\kappa\circ \kappa'^{-1}$ defined on $\kappa'(\tilde U_\kappa\cap \tilde U_{\kappa'})$. Then for each $f\in \mathscr{D}(\kappa(U_{\kappa}\cap   U_{\kappa'}))$ and each multi-index $I_p=\{i_1,\cdots,i_p\}$ fixed, one infers
\begin{align}
\label{currentlocalcap1}
 \langle a_{\kappa;I_p},f\mathrm{d}x_\kappa\rangle
=&\frac{1}{(n-p)!}\delta_{I_p,I_p^c}\langle a_\kappa,f\mathrm{d}x_\kappa^{I_p^c}\rangle\\
=&\frac{1}{(n-p)!}\delta_{I_p,I_p^c}\langle T,(f \circ\kappa) \kappa^*\left(\mathrm{d}x_\kappa^{I_p^c}\right)\rangle\nonumber\\
=&\frac{1}{(n-p)!}\delta_{I_p,I_p^c}\langle T,((f \circ\psi)\circ\kappa') \kappa'^*\left(\psi^*\mathrm{d}x_\kappa^{I_p^c}\right)\rangle\nonumber\\
=&\frac{1}{(n-p)!}\delta_{I_p,I_p^c} \langle T ,\left((f\circ \psi) \frac{\partial x_\kappa^{i_{p+1}}}{\partial x_{\kappa'}^{j_{p+1}}}\cdots
\frac{\partial x_\kappa^{i_n}}{\partial x_{\kappa'}^{j_n}}\right)\circ \kappa'\kappa'^*\left(\mathrm{d}x_{\kappa'}^{J_p^c}\right)\rangle \nonumber\\
=&\frac{1}{(n-p)!}\delta_{I_p,I_p^c} \langle a_{\kappa'} , (f\circ \psi) \frac{\partial x_\kappa^{i_{p+1}}}{\partial x_{\kappa'}^{j_{p+1}}}\cdots
\frac{\partial x_\kappa^{i_n}}{\partial x_{\kappa'}^{j_n}} \mathrm{d}x_{\kappa'}^{J_p^c} \rangle \nonumber\\
=&\frac{1}{p!(n-p)!}\delta_{I_p,I_p^c}\delta_{J_p,J_p^c}\langle  a_{\kappa';J_p},(f\circ \psi) \frac{\partial x_\kappa^{i_{p+1}}}{\partial x_{\kappa'}^{j_{p+1}}}\cdots
\frac{\partial x_\kappa^{i_n}}{\partial x_{\kappa'}^{j_n}} \mathrm{d}x_{\kappa'}\rangle, \nonumber
\end{align}
where $I_p^c$ (resp. $J_p^c$) is the supplementary set of $I_p$ (resp. $J_p$) such that $I_p\cup I_p^c=\{1,\cdots,n\}$ (resp. $J_p\cup J_p^c=\{1,\cdots,n\}$).

A direct calculation, together with \eqref{hormanderequ6.1.1}, yields that
\begin{equation}
\label{currentlocalcap2}
\left(\psi^*a_{\kappa;I_p}\right)(f\circ\psi)
=\left(\det (\psi')\right)^{-1}\langle a_{\kappa;I_p},f\circ\psi\circ \psi^{-1}\rangle
=\left(\det (\psi')\right)^{-1}\langle a_{\kappa;I_p},f\rangle.
\end{equation}
 It follows from \eqref{currentlocalcap1} and \eqref{currentlocalcap2} that
\begin{equation}
\label{pcurrentlocalrelationorigin}
\frac{1}{p!(n-p)!}\delta_{I_p,I_p^c}\delta_{J_p,J_p^c}a_{\kappa';J_p}\frac{\partial x_\kappa^{i_{p+1}}}{\partial x_{\kappa'}^{j_{p+1}}}\cdots
\frac{\partial x_\kappa^{i_n}}{\partial x_{\kappa'}^{j_n}}=\left(\det ((\kappa\circ \kappa'^{-1})')\right) (\kappa\circ \kappa'^{-1})^*a_{\kappa;I_p},
\end{equation}
which is equivalent to \eqref{hormanderequ6.3.3current}. Indeed, it follows from \eqref{pcurrentlocalrelationorigin} that
\begin{align*}
&\left(\det ((\kappa\circ \kappa'^{-1})')\right) (\kappa\circ \kappa'^{-1})^*a_{\kappa;I_p}\frac{\partial x_\kappa^{i_{1}}}{\partial x_{\kappa'}^{k_{1}}}\cdots
\frac{\partial x_\kappa^{i_p}}{\partial x_{\kappa'}^{k_p}}\\
=&\frac{1}{p!(n-p)!}\delta_{I_p,I_p^c}\delta_{J_p,J_p^c}a_{\kappa';J_p}
\frac{\partial x_\kappa^{i_{1}}}{\partial x_{\kappa'}^{k_{1}}}\cdots
\frac{\partial x_\kappa^{i_p}}{\partial x_{\kappa'}^{k_p}}
\frac{\partial x_\kappa^{i_{p+1}}}{\partial x_{\kappa'}^{j_{p+1}}}\cdots
\frac{\partial x_\kappa^{i_n}}{\partial x_{\kappa'}^{j_n}}\\
=&\frac{1}{p!(n-p)!} \delta_{J_p,J_p^c}a_{\kappa';J_p}\delta_{K_p,J_p^c}
\left(\det ((\kappa\circ \kappa'^{-1})')\right)\\
=&a_{\kappa';K_p}\left(\det ((\kappa\circ \kappa'^{-1})')\right),
\end{align*}
as desired.

For `only if' direction, we choose a partition of unitary $1=\sum_{j}\chi_{\kappa_j}$ with $\chi_{\kappa_j}\in\mathscr{D}(\tilde U_{\kappa_j})$ for some $\kappa_j \in\mathscr{A}$. A direct calculation yields that the linear form $T$ defined by
\begin{equation}
\langle T,\varphi\rangle:=\sum_j\langle a_{\kappa_j},(\kappa_j^{-1})^*(\chi_{\kappa_j}\varphi)\rangle,
\quad\forall\;\varphi\in\mathscr{D}^{n-p}(X)
\end{equation}
is continuous in $\mathscr{D}^{n-p}(X)$. Let $T_\kappa\in \mathscr{D}_{n-p}'(\tilde U_\kappa)$ defined by \eqref{localcurrentdefn}. Then for each  $\varphi\in\mathscr{D}^{n-p}(\tilde U_\kappa)$, we have
\begin{align*}
\langle T_\kappa,\varphi\rangle
=&\langle T , \kappa ^*\varphi\rangle\\
=&\sum_j\langle a_{\kappa_j},(\kappa_j^{-1})^*(\chi_{\kappa_j}\kappa^*\varphi)\rangle\\
=&\sum_j\langle a_{\kappa_j}, (\chi_{\kappa_j}\circ\kappa_j^{-1})(\kappa\circ \kappa_j^{-1})^* \varphi \rangle\\
=&\sum_j\langle a_{\kappa_j}, (\kappa\circ \kappa_j^{-1})^* ((\chi_{\kappa_j}\circ\kappa^{-1})\varphi) \rangle\\
=&\sum_j\langle a_{\kappa},  ((\chi_{\kappa_j}\circ\kappa^{-1})\varphi) \rangle\\
=&\langle a_{\kappa},   \varphi\rangle,
\end{align*}
where we use the fact that Equation \eqref{hormanderequ6.3.3current} is equivalent to the fact that
\begin{equation*}
\langle a_{\kappa'},(\kappa\circ\kappa'^{-1})^*\varphi\rangle=\langle a_{\kappa},\varphi\rangle,
\quad\forall\;\varphi\in\mathscr{D}^{n-p}(\kappa(U_\kappa\cap U_{\kappa'})).
\end{equation*}
\end{proof}
It follows from Definition \ref{hormanderdefn6.3.3} and Theorem \ref{hormanderdefn6.3.3revisedforcurrent} that a distribution $u\in\mathscr{D}'(X)$ on $X$ can be viewed as a current with degree $0$.

Many of the operations available for differential forms can
be extended to currents by simple duality arguments.
 In general, if $A:\;\bigoplus\mathscr{D}^{p}(X)\to \bigoplus\mathscr{D}^{p}(X)$ is a map  of vector spaces which is continuous on $\bigoplus\mathscr{D}^{p}(K)$ for each compact set $K$, then it is possible to extend $A$ to a mapping currents.

Let $T\in{}^s\mathscr{D}'^p(X)={}^s\mathscr{D}'_{m-p}(X)$. The exterior derivative
$$
\mathrm{d} T\in {}^{s+1}\mathscr{D}'^{p+1}(X)={}^{s+1}\mathscr{D}'_{m-p-1}(X)
$$
is defined by
\begin{equation}
\label{equ1-2.6}
\langle \mathrm{d} T,\varphi\rangle=(-1)^{p+1}\langle T,\mathrm{d} \varphi\rangle,\quad \forall\;\varphi\in{}^{s+1}\mathscr{D}^{m-p-1}(X).
\end{equation}
The continuity of the linear form $\mathrm{d} T$ on ${}^{s+1}\mathscr{D}^{m-p-1}(X)$ follows from the continuity of the map $\mathrm{d}:\,{}^{s+1}\mathscr{D}^{m-p-1}(K)\to {}^{s}\mathscr{D}^{m-p}(K)$.

Let $T\in{}^s\mathscr{D}'^{p+1}(X)={}^s\mathscr{D}'_{m-p-1}(X)$ and $\xi\in\mathfrak{X}(X)$ a vector field. The interior product of $\xi$ and $T$
$$
i_\xi T\in {}^{s+1}\mathscr{D}'^{p }(X)={}^{s+1}\mathscr{D}'_{m-p}(X)
$$
is defined (see for example \cite{touzettoulouse2010}) by
\begin{equation}
\label{equ1-2.6innerproduct}
\langle i_\xi T,\varphi\rangle=(-1)^{p}\langle T,i_\xi\varphi\rangle,\quad \forall\;\varphi\in{}^{s+1}\mathscr{D}^{m-p}(X).
\end{equation}
The continuity of the linear form $i_\xi T$ on ${}^{s+1}\mathscr{D}^{m-p}(X)$ follows from the continuity of the map $i_\xi:\,{}^{s+1}\mathscr{D}^{m-p}(K)\to {}^{s}\mathscr{D}^{m-p+1}(K)$.

Let $T\in{}^s\mathscr{D}'^p(X)={}^s\mathscr{D}'_{m-p}(X)$. Then since the Lie derivative $\mathcal{L}_\xi=\mathrm{d}\circ i_\xi+i_\xi\circ \mathrm{d}$, we can define the Lie derivative of a current $T$
$$
\mathcal{L}_\xi T\in {}^{s+1}\mathscr{D}'^{p}(X)={}^{s+1}\mathscr{D}'_{m-p}(X)
$$
by
\begin{equation}
\label{equ1-2.6liederivative}
\langle \mathcal{L}_\xi T,\varphi\rangle=-\langle T,\mathcal{L}_\xi \varphi\rangle,\quad \forall\;\varphi\in{}^{s+1}\mathscr{D}^{m-p}(X).
\end{equation}
\begin{lem}
[Transverse Poincar\'e Lemma]
Let $T\in {}^s\mathscr{D}'_{n-p}(U_1\times U_2),$ where $ U_1\subset\mathbb{R}^k$ and $ U_2\subset\mathbb{R}^n$ are open sets. We denote by $t=(t^1,\cdots,t^k)$ $($resp. $x=(x^1,\cdots,x^n)$$)$ the coordinates of $\mathbb{R}^k$ $($resp. $\mathbb{R}^n$$)$. If $\mathrm{d}T=0$ and $i_{\partial/\partial t^i}T=\mathcal{L}_{\partial/\partial t^i}T=0$ with $1\leq i\leq k$, then there exist a $S \in {}^s\mathscr{D}'_{n-p+1}(U_1\times U_2)$ and  a smooth form $\Theta\in{}^s\mathscr{E}^{p}(U_1\times U_2)$ such that $T=\Theta+\mathrm{d}S$ and $i_{\partial/\partial t^i}S=\mathcal{L}_{\partial/\partial t^i}S=i_{\partial/\partial t^i}\Theta=\mathcal{L}_{\partial/\partial t^i}\Theta=\mathrm{d}\Theta=0$ with $1\leq i\leq k$.
\end{lem}
\begin{proof}
Since $i_{\partial/\partial t^i}T=\mathcal{L}_{\partial/\partial t^i}T=0$ with $1\leq i\leq k$, we have
\begin{equation}
\label{tpoincarelemmalocalt}
T=\sum_{1\leq i_1<\cdots<i_p\leq n}T_{i_1,\cdots,i_p}\mathrm{d}x^{i_1}
\wedge\cdots\wedge \mathrm{d}x^{i_p}
\end{equation}
with $\frac{\partial T_{i_1,\cdots,i_p}}{\partial t^\ell}=0$ for $1\leq \ell\leq k$.
It follows from \eqref{tpoincarelemmalocalt} Proposition \ref{ramondexercise4.2.9}  that
\begin{equation}
T_{i_1,\cdots,i_p}=1\otimes \tilde T_{i_1,\cdots,i_p},
\end{equation}
where $\tilde T_{i_1,\cdots,i_p}\in \mathscr{D}'(U_2)$
and that
\begin{equation}
\label{tpoincarelemmadtildet=0}
\mathrm{d}_{\mathbb{R}^n}\tilde T=0,
\end{equation}
where
\begin{equation*}
\tilde T=\sum_{1\leq i_1<\cdots<i_p\leq n}\tilde T_{i_1,\cdots,i_p}\mathrm{d}x^{i_1}
\wedge\cdots\wedge \mathrm{d}x^{i_p}.
\end{equation*}
It follows the standard Poincar\'e Lemma (see for example \cite[Section 2.D.4]{demaillybook1}) and \eqref{tpoincarelemmadtildet=0} that
\begin{equation}
\tilde T= \Theta +\mathrm{d}_{\mathbb{R}^n}\tilde S
\end{equation}
where $\tilde S\in{}^s\mathscr{D}'_{n-p+1}(U_2)$ given by
$$
\tilde S=\sum_{1\leq i_1<\cdots<i_{p-1}\leq n}\tilde{S}_{i_1,\cdots,i_{p-1}}
\mathrm{d}x^{i_1}
\wedge\cdots\wedge \mathrm{d}x^{i_{p-1}}
$$
and $\mathrm{d}_{\mathbb{R}^n}$ closed $p$ form $ \Theta\in \mathscr{E}^{p}(U_2)$ given by
$$
 \Theta=\sum_{1\leq i_1<\cdots<i_p\leq n}\Theta_{i_1,\cdots,i_p}\mathrm{d}x^{i_1}
\wedge\cdots\wedge \mathrm{d}x^{i_p}
$$
with $\frac{\partial \Theta_{i_1,\cdots,i_p}}{\partial t^\ell}=0$ for $1\leq \ell\leq k$.
Hence $\Theta$ and $S$ given by
$$
 S=\sum_{1\leq i_1<\cdots<i_{p-1}\leq n}(1\otimes\tilde{S}_{i_1,\cdots,i_{p-1}})
\mathrm{d}x^{i_1}
\wedge\cdots\wedge \mathrm{d}x^{i_{p-1}}.
$$
are the required currents from Proposition  \ref{ramondexercise4.2.7}.
\end{proof}
 Now
we have $\mathrm{d}\Theta=0$ and we deduce from applying the usual Poincar\'{e} lemma (see for example \cite[(1.22)]{demaillybook1}  to $\Theta$ that
\begin{thm}
\label{tthmbasiccurrentcohomology}
Let $T\in {}^s\mathscr{D}'_{n-p}(U_1\times U_2),$ where $ U_1\subset\mathbb{R}^k$ and $ U_2\subset\mathbb{R}^n$ are open sets and in addition $U_2$ is  star-shaped open. We denote by $t=(t^1,\cdots,t^k)$ $($resp. $x=(x^1,\cdots,x^n)$$)$ the coordinates of $\mathbb{R}^k$ $($resp. $\mathbb{R}^n$$)$. If $\mathrm{d}T=0$ and $i_{\partial/\partial t^i}T=\mathcal{L}_{\partial/\partial t^i}T=0$ with $1\leq i\leq k$, then there exists a $S \in {}^s\mathscr{D}'_{n-p+1}(U_1\times U_2)$ such that $T=\mathrm{d}S$ and $i_{\partial/\partial t^i}S=\mathcal{L}_{\partial/\partial t^i}S=0$ with $1\leq i\leq k$.
\end{thm}
Similar argument gives  the transverse  Dolbeault-Grothendieck lemma as follows.
\begin{lem}
[Transverse Dolbeault-Grothendieck Lemma]
\label{tdglemma}
Let $u\in {}^s\mathscr{E}^{p,q}(U_1\times U_2,\mathbb{C})$ $($resp. $T\in {}^s\mathscr{D}'_{n-p,n-q}(U_1\times U_2)$$),$ where $ U_1\subset\mathbb{R}^k$ and $ U_2\subset\mathbb{C}^n$ containing $0$ are open sets. We denote by $t=(t^1,\cdots,t^k)$ $($resp. $z=(z^1,\cdots,z^n)$$)$ the coordinates of $\mathbb{R}^k$ $($resp. $\mathbb{C}^n$$)$. If $\bar\partial_{\mathbb{C}^n} u=0$ and $i_{\partial/\partial t^i}u=\mathcal{L}_{\partial/\partial t^i}u=0$ with $1\leq i\leq k,$ then we have
\begin{enumerate}
\item If $q=0,$ then $u=\sum_{1\leq i_1<\cdots<i_p\leq n}u_{i_1,\cdots,i_p}\mathrm{d}z^{i_1}\wedge\cdots\wedge\mathrm{d}z^{i_p}$ where $u_{i_1,\cdots,i_p}$ is independent of $t$ and holomorphic for $z$.
\item If $q\geq 1,$ then there exists a neighborhood $V_2\subset U_2$ of $0$ and a form $v\in{}^s\mathscr{E}^{p,q-1}(U_1\times V_2,\mathbb{C})$ $($resp. $v\in {}^s\mathscr{D}'_{n-p,n-q+1}(U_1\times V_2)$$)$ such that $u=\bar\partial_{\mathbb{C}}v$ on $U_1\times V_2$ and $i_{\partial/\partial t^i}v=\mathcal{L}_{\partial/\partial t^i}v=0$ with $1\leq i\leq k$.
\end{enumerate}
\end{lem}
\end{appendix}


\begin{thebibliography}{10}

\bibitem{aubin78}
Thierry Aubin.
\newblock {\'E}quations du type monge-amp\`ere sur les vari\'{e}t\'{e}s
  k\"ahleriennes compactes.
\newblock {\em Bulletin des Sciences Math\'{e}matiques}, 102:63--95, 1978.

\bibitem{baragliahekmati2018}
David Baraglia and Pedram Hekmati.
\newblock A foliated {H}itchin-{K}obayashi correspondence.
\newblock 2018.
\newblock arXiv:1802.09699.

\bibitem{barreelkacimi-alouifoliations}
Raymond Barre and Aziz El~Kacimi-Alaoui.
\newblock Foliations.
\newblock In Franki~J.E. Dillen and Leopold C.~A. Verstraelen, editors, {\em
  Handbook of Differential Geometry}, volume~2, chapter~2, pages 35--78. North
  Holland, 2006.

\bibitem{bs2010}
Indrani Biswas and Georg Schumacher.
\newblock Vector bundles on {S}asakian manifolds.
\newblock {\em Advances in Theoretical and Mathematical Physics}, 14:541--561,
  2010.

\bibitem{brendleetx2017}
Simon Brendle, Kyeongsu Choi, and Panagiota Daskalopoulos.
\newblock Asymptotic behavior of flows by powers of the {G}aussian curvature.
\newblock {\em Acta Mathematica}, 17(1):1--16, 2017.

\bibitem{chernclass}
Shiing-Shen Chern.
\newblock Vector bundle with a connection.
\newblock In Shiing-Shen Chern, editor, {\em Global Differential Geometry},
  volume~27 of {\em MAA Studies in Mathematics}, pages 1--26. Mathematical
  Association of America, 1989.

\bibitem{cs18}
Tristan~C. Collins and G\'abor Sz\'ekelyhidi.
\newblock {K}-semistability for irregular {S}asakian manifolds.
\newblock {\em Journal of Differential Geometry}, 109(1):81--109, 2018.

\bibitem{cs15}
Tristan~C. Collins and G\'abor Sz\'ekelyhidi.
\newblock Sasaki-{E}instein metrics and {K}-stability.
\newblock {\em Geometry \& Topology}, 23:1339--1413, 2019.

\bibitem{craioveanuputa1980}
Mircea Craioveanu and Mircea Puta.
\newblock Cohomology on a {R}iemannian foliated manifold with coefficients in
  the sheaf of germs of foliated currents.
\newblock {\em Mathematische Nachrichten}, 99:43--53, 1980.

\bibitem{derham55}
Georges de~Rham.
\newblock {\em Vari\'et\'es diff\'erentiables}.
\newblock Hermann, Paris, 1955.

\bibitem{dem92b}
Jean-Pierre Demailly.
\newblock Regularization of closed positive currents and intersection theory.
\newblock {\em Journal of Algebraic Geometry}, 1(3):361--409, 1992.

\bibitem{demaillybook1}
Jean-Pierre Demailly.
\newblock {C}omplex {A}nalytic and {D}ifferential {G}eometry.
\newblock e-book available on the author's webpage,
  \url{https://www-fourier.ujf-grenoble.fr/~demailly/manuscripts/agbook.pdf},
  June 2012.

\bibitem{DTjdg} Simone Diverio and Stefano Trapani.
\newblock Quasi-negative holomorphic sectional curvature and positivity of the canonical bundle,
\newblock{\em Journal of Differential Geometry} 111 (2019), no. 2, 303-314

\bibitem{elk90}
Aziz El~Kacimi-Alaoui.
\newblock Op\'erateurs transversalement elliptiques sur un feuilletage
  riemannien et applications.
\newblock {\em Compositio Mathematica}, 73(1):57--106, 1990.

\bibitem{elk2014}
Aziz El~Kacimi-Alaoui.
\newblock Fundaments of foliation theory.
\newblock In Jes\'{u}s~\'{A}lvarez L\'{o}pez and Marcel Nicolau, editors, {\em
  Foliations: dynamics, geometry and topology}, Advanced Courses in
  Mathematics. CRM Barcelona, pages 41--86. Birkh\"auser/Springer, 2014.

\bibitem{ekhe86}
Aziz El~Kacimi-Alaoui and Gilbert Hector.
\newblock D\'ecomposition de {H}odge basique pour un feuilletage riemannien.
\newblock {\em Annales de l'institut Fourier $($Grenoble$)$}, 36(3):207--227,
  1986.

\bibitem{evanscpam1982}
Lawrence~Craig Evans.
\newblock Classical solutions of fully nonlinear, convex, second-order elliptic
  equations.
\newblock {\em Communications on Pure and Applied Mathematics}, 35(3):333--363,
  1982.

\bibitem{fengzhengsasaki1}
Ke~Feng and Tao Zheng.
\newblock Transverse fully nonlinear equations on {S}asakian manifolds and
  applications.
\newblock {\em Advances in Mathematics}, 357:106830, 2019.
\newblock https://doi.org/10.1016/j.aim.2019.106830.

\bibitem{fow09}
Akito Futaki, Hajimi Ono, and Guofang Wang.
\newblock Transverse {K}\"ahler geometry of {S}asaki manifolds and toric
  {S}asaki-{E}instein manifolds.
\newblock {\em Journal of Differential Geometry}, 83(3):585--636, 2009.

\bibitem{gt1998}
David Gilbarg and Neil~Sidney Trudinger.
\newblock {\em Elliptic Partial Differential Equations of Second Order}.
\newblock Classical in Mathematics. Springer-Verlag Berlin Heidelberg, 2001.

\bibitem{bookgz17}
Vincent Guedj and Ahmed Zeriahi.
\newblock {\em Degenerate Complex Monge--Amp{\`e}re Equations}, volume~26 of
  {\em EMS Tracts in Mathematics}.
\newblock European Mathematical Society, 2017.
\newblock Winner of the 2016 EMS Monograph Award.

\bibitem{heafligerjdg1980}
Andr\'e Haefliger.
\newblock Some remarks on foliations with minimal leaves.
\newblock {\em Journal of Differential Geometry}, 15:269--284, 1980.

\bibitem{hl18}
Weiyong He and Jun Li.
\newblock Geometrical pluripotential theory on {S}asaki manifolds.
\newblock {\em The Journal of Geometric Analysis}, 31:1093--1179, 2021.
\newblock arXiv:1803.00687.

\bibitem{hs15}
Weiyong He and Song Sun.
\newblock The generalized {F}rankel conjecture in {S}asaki geometry.
\newblock {\em International Mathematics Research Notices. IMRN},
  2015(1):99--118, 2015.

\bibitem{hs16}
Weiyong He and Song Sun.
\newblock Frankel conjecture and {S}asaki geometry.
\newblock {\em Advances in Mathematics}, 291:912--960, 2016.

\bibitem{hormander}
Lars H\"ormander.
\newblock {\em The Analysis of Linear Partial Differential Operators {\rm I:}
  Distribution Theory and Fourier Analysis}, volume 256 of {\em Grundlehren der
  Mathematischen Wissenschaften, A Series of Comprehensive Studies in
  Mathematics}.
\newblock Springer-Verlag, 2nd edition, 1989.

\bibitem{kambertondeur1975}
Franz~W. Kamber and Phillip Tondeur.
\newblock {\em Foliated Bundles and Characteristic Classes}, volume 493 of {\em
  Lecture Notes in Mathematics}.
\newblock Springer-Verlag, 1975.

\bibitem{krylov1982}
Nicolai~Vladimirovich Krylov.
\newblock Boundedly nonhomogeneous elliptic and parabolic equations.
\newblock {\em Izvestiya Akademii Nauk SSSR. Seriya Matematicheskaya},
  46(3):487--523, 1982.

\bibitem{kupfermantvs}
Raz Kupferman.
\newblock Topological vector spaces.
\newblock
  \url{http://www.ma.huji.ac.il/~razk/iWeb/My_Site/Teaching_files/TVS.pdf}.

\bibitem{lelong57}
Pierre Lelong.
\newblock Int\'egration sur un ensemble analytique complexe.
\newblock {\em Bulletin de la Soci\'et\'e Math\'ematique de France},
  85:239--262, 1957.

\bibitem{masa1992}
Xos\'e Masa.
\newblock Duality and minimality in {R}iemannian foliations.
\newblock {\em Commentarii Mathematici Helvetici}, 67(1):17--27, 1992.

\bibitem{ps09}
Duong~Hong Phong and Jacob Sturm.
\newblock The {D}irichlet problem for degenerate complex {M}onge-{A}mp\`ere
  equations.
\newblock {\em Communications in Analysis and Geometry}, 18(1):145--170, 2010.

\bibitem{ramond2015}
Thierry Ramond.
\newblock {D}istributions and {P}artial {D}ifferential {E}quations.
\newblock
  \url{https://www.imo.universite-paris-saclay.fr/~ramond/docs/cours/distributionsedp.pdf}.

\bibitem{roydencmh1980}
Halsey~L. Royden.
\newblock The {A}hlfors-{S}chwarz lemma in several complex variables.
\newblock {\em Commentarii Mathematici Helvetici}, 55(4):547--558, 1980.

\bibitem{rudinrealanalysis}
Walter Rudin.
\newblock {\em Real and Complex Analysis}.
\newblock Mathematics series. McGraw-Hill Sciences, third edition, 1987.

\bibitem{rummlercmh1979}
Hansklaus Rummler.
\newblock Quelques notions simples en g\'eom\'etrie riemannienne et leurs
  applications aux feuilletages compacts.
\newblock {\em Commentarii Mathematici Helvetici volume}, 54(2):224--239, 1979.

\bibitem{sullivancmh1979}
Dennis Sullivan.
\newblock A homological characterization of foliations consisting of minimal
  surfaces.
\newblock {\em Commentarii Mathematici Helvetici}, 54(2):218--223, 1979.

\bibitem{gaborjdg}
G{\'a}bor Sz{\'e}kelyhidi.
\newblock Fully non-linear elliptic equations on compact {H}ermitian manifolds.
\newblock {\em Journal Differential Geometry}, 109(2):337--378, 2018.

\bibitem{stw1503}
G{\'a}bor Sz\'ekelyhidi, Valention Tosatti, and Ben Weinkove.
\newblock {G}auduchon metrics with prescribed volume form.
\newblock {\em Acta Mathematica}, 219(1):181--211, 2017.

\bibitem{tosattikawa}
Valentino Tosatti.
\newblock {KAWA} lecture notes on the {K}\"ahler-{R}icci flow.
\newblock {\em Annales de la Facult\'e des Sciences de Toulouse.
  Math\'ematiques. S\'erie 6}, 27(2):285-376, 2018.

\bibitem{tosattinakamaye}
Valentino Tosatti.
\newblock Nakamaye's theorem on complex manifolds.
\newblock In {\em Algebraic geometry$:$ Salt Lake City 2015}, volume 97.1 of
  {\em Proceedings of Symposia in Pure Mathematics}, pages 633--655. American
  Mathematical Society, Providence, RI, 2018.

\bibitem{twwycvpde}
Valentino Tosatti, Yu~Wang, Ben Weinkove, and Xiaokui Yang.
\newblock ${C}^{2,\alpha}$ estimates for nonlinear elliptic equations in
  complex and almost complex geometry.
\newblock {\em Calculus of Variations and Partial Differential Equations},
  54(1):431--453, September 2015.

\bibitem{twjams10}
Valentino Tosatti and Ben Weinkove.
\newblock The complex {M}onge-{A}mp{\`e}re equation on compact {H}ermitian
  manifolds.
\newblock {\em Journal of the American Mathematical Society}, 23(4):1187--1195,
  2010.

\bibitem{twjams}
Valentino Tosatti and Ben Weinkove.
\newblock The {M}onge-{A}mp{\`e}re equation for $(n-1)$-plurisubharmonic
  functions on a compact {K}{\"a}hler manifold.
\newblock {\em Journal of the American Mathematical Society}, 30(2):311--346,
  2017.

\bibitem{twarxiv1906}
Valentino Tosatti and Ben Weinkove.
\newblock The complex {M}onge-{A}mp\`ere equation with a gradient term.
\newblock {\em Pure and Applied Mathematics Quarterly}, 17(3):1005--1024, 2019.
\newblock special issue for D.H. Phong's 65th birthday.

\bibitem{twcrelle}
Valentino Tosatti and Ben Weinkove.
\newblock Hermitian metrics, $(n-1,n-1)$-forms and {M}onge-{A}mp\`ere
  equations.
\newblock {\em Journal f\"ur die Reine und Angewandte Mathematik $($Crelle's
  Journal$)$}, 2019(755):67--101, 2019.

\bibitem{tosattiyangjdg}
Valentino Tosatti and Xiaokui Yang.
\newblock An extension of a theorem of {W}u-{Y}au.
\newblock {\em Journal of Differential Geometry}, 107(3):573--579, 2017.

\bibitem{touzettoulouse2010}
Fr\'{e}d\'{e}ric Touzet.
\newblock Structure des feuilletages k\"ahleriens en courbure semi-n\'egative.
\newblock {\em Annales de la Facult\'e des Sciences de Toulouse.
  Math\'ematiques. S\'erie 6}, 19(3-4):865--886, 2010.

\bibitem{trudingertransactionams1983}
Neil~Sidney Trudinger.
\newblock Fully nonlinear, uniformly elliptic equations under natural structure
  conditions.
\newblock {\em Transactions of the American Mathematical Society},
  278(2):751--769, 1983.

\bibitem{uhlenbeckyau}
Karen~K. Uhlenbeck and Shing-Tung Yau.
\newblock On the existence of {H}ermitian-{Y}ang-{M}ills connections in stable
  vector bundles.
\newblock {\em Communications on Pure and Applied Mathematics},
  39(S1):S257--S293, 1986.

\bibitem{vaisman}
Izu Vaisman.
\newblock {\em Cohomology and Differential Forms}, volume~33 of {\em Dover
  Books on Mathematics}.
\newblock Dover Publications, 2nd edition, 2016.


\bibitem{vancoevering2016}
Craig van Coevering.
\newblock Monge-{A}mp\`ere operators, energy functionals, and uniqueness of
  {S}asaki-extremal metrics.
\newblock arXiv:1511.09167, Version 2.

\bibitem{weinkovekrfpark}
Ben Weinkove.
\newblock The {K}\"ahler-{R}icci flow on compact {K}\"ahler manifolds.
\newblock In Hubert~L. Bray, Greg Galloway, Rafe Mazzeo, and Natasa Sesum,
  editors, {\em Geometric Analysis}, volume~22 of {\em IAS/Park City
  Mathematics Series}, pages 53--108. American Mathematical Society,
  Providence,RI, 2016.

\bibitem{wongwuyauproceedingams2012}
Pit-Mann Wong, Damin Wu, and Shing-Tung Yau.
\newblock Picard number, holomorphic sectional curvature, and ampleness.
\newblock {\em Proceedings of the American Mathematical Society},
  140(2):621--626, 2012.

\bibitem{wuyauinventiones}
Damin Wu and Shing-Tung Yau.
\newblock Negative holomorphic curvature and positive canonical bundle.
\newblock {\em Inventiones Mathematicae}, 204(2):595--604, 2016.

\bibitem{wuyaucag}
Damin Wu and Shing-Tung Yau.
\newblock A remark on our paper `negative holomorphic curvature and positive
  canonical bundle'.
\newblock {\em Communications in Analysis and Geometry}, 24(4):901--912, 2016.

\bibitem{wuyauzhengmrl2009}
Damin Wu, Shing-Tung Yau, and Fangyang Zheng.
\newblock A degenerate {M}onge-{A}mp\`ere equation and the boundary classes of
  {K}\"ahler cones.
\newblock {\em Mathematical Research Letters}, 16(2):365--374, 2009.

\bibitem{wuzhangscm}
Di~Wu and Xi~Zhang.
\newblock Higgs bundles over foliated manifolds.
\newblock {\em Science China Mathematics}, 64(2):399--420, 2020.

\bibitem{YZtams} Xiaokui Yang and Fangyang Zheng.
\newblock On real bisectional curvature for Hermitian manifolds, \newblock {\em Transactions of the American Mathematical Society} 371 (2019), no. 4, 2703-2718


\bibitem{yauajm1978}
Shing-Tung Yau.
\newblock A general {S}chwarz lemma for {K}\"ahler manifolds.
\newblock {\em American Journal of Mathematics}, 100(1):197--203, 1978.

\bibitem{yau1978}
Shing-Tung Yau.
\newblock On the {R}icci curvature of a compact {K}\"ahler manifold and the
  complex {M}onge-{A}mp\`ere equation,\textrm{I}.
\newblock {\em Communications on Pure and Applied Mathematics}, 31(3):339--411,
  1978.

\bibitem{Zmz} Yashan Zhang.
 \newblock Holomorphic sectional curvature, nefness and Miyaoka-Yau type inequality,
 \newblock {\em Mathematische Zeitschrift} 298, 953-974 (2021)

\bibitem{ZZ} Yashan Zhang and Tao Zheng.
\newblock On almost quasi-negative holomorphic sectional curvature,
\newblock preprint, arXiv:2010.01314

\end{thebibliography}

\end{document}